\pdfoutput=1
\documentclass[12pt,twoside]{article}

    \usepackage{graphicx}
    \usepackage{styleset}
    \usepackage{macros}
    \let\subsubsection\subparagraph

    \title  {Exact triangles for $\SO(3)$ instanton homology of webs}

    \author {P. B. Kronheimer and T. S. Mrowka%
      \thanks{%
        The work of the first author was supported by the National
        Science Foundation through NSF grants DMS-0904589 and
        DMS-1405652. The work of the second author was supported by
        NSF grants DMS-0805841 and DMS-1406348.}}

    \address {Harvard University, Cambridge MA 02138 \\
              Massachusetts Institute of Technology, Cambridge MA 02139}

\usepackage{xr}
\begin{document}

\maketitle

\section{Introduction}

Let $K \subset \R^{3}$ be an unoriented web, i.e. an embedded
trivalent graph whose local model at the vertices is that of three
arcs meeting with distinct tangent directions. 
In a previous paper \cite{KM-jsharp}, the authors defined an invariant
$\Jsharp(K)$ for such webs, as an $\SO(3)$ instanton homology with
coefficients in the field $\F=\Z/2$. 
This instanton homology is functorial for \emph{foams}, which are singular
cobordisms between webs. The construction of $\Jsharp(K)$ closely
resembles an  invariant $\Isharp(K)$ defined earlier for knots and
links in \cite{KM-unknot}. Knots and links are webs without vertices;
but even for these, $\Isharp(K)$ and $\Jsharp(K)$ are different,
because $\Isharp(K)$ was defined using $\SU(2)$ representation
varieties, while $\Jsharp(K)$ uses $\SO(3)$. 
Our conventions and definitions are briefly
recalled in section~\ref{sec:recap}.

This paper is a continuation of \cite{KM-jsharp} and establishes
a type of skein relation (an exact triangle) for $\Jsharp$. 
The main result concerns three
webs $L_{2}$, $L_{1}$, $L_{0}$ which differ only inside a ball, as
shown:
\begin{equation}\label{eq:intro-L-pics}
        L_{2} = \mathfig{0.079}{figures/Graph-Xbar}, \qquad
        L_{1} = \mathfig{0.079}{figures/Graph-H}, \qquad
        L_{0} = \mathfig{0.079}{figures/Graph-I}.
\end{equation}
There are standard foam cobordisms between these (see
section~\ref{sec:statements} for a fuller description):
\[
      \cdots \longrightarrow \mathfig{0.079}{figures/Graph-Xbar}
       \longrightarrow \mathfig{0.079}{figures/Graph-H}
          \longrightarrow \mathfig{0.079}{figures/Graph-I}
                \longrightarrow \mathfig{0.079}{figures/Graph-Xbar}
                     \longrightarrow\cdots
\]
We then have:

\begin{theorem}\label{thm:intro-three-bar}
    The sequence of $\F$-vector spaces obtained by applying $\Jsharp$
    to the above sequence of webs and foams is exact:
  \[
        \cdots \longrightarrow \Jsharp(L_{2}) \longrightarrow
                 \Jsharp(L_{1}) \longrightarrow
                  \Jsharp(L_{0}) \longrightarrow
                    \Jsharp(L_{2}) \longrightarrow \cdots
  \]
\end{theorem}

There is a variant of this exact triangle. Consider three webs
differing in the ball as in the following diagrams:
\[
        K_{2} = \mathfig{0.079}{figures/Graph-X}, \qquad
        K_{1} = \mathfig{0.079}{figures/Graph-Res1}, \qquad
        K_{0} = \mathfig{0.079}{figures/Graph-Res0}.
\]
Again, there are standard cobordisms between these.
In \cite{KM-unknot}, we established an exact triangle in $\Isharp$
relating these three. The corresponding sequence of vector spaces
$\Jsharp(K_{i})$ do \emph{not} form an exact triangle. Instead, there
is an exact triangle involving $L_{i+2}$, $K_{i+1}$ and $K_{i}$, for
each $i$ (with the indices interpreted cyclically modulo $3$). Thus:

\begin{theorem}\label{thm:intro-one-bar}
    For each $i=0, 1, 2$, we have an exact sequence of $\F$-vector
    spaces,
  \[
        \cdots \longrightarrow \Jsharp(L_{i+2}) \longrightarrow
                 \Jsharp(K_{i+1}) \longrightarrow
                  \Jsharp(K_{i}) \longrightarrow
                    \Jsharp(L_{i+2}) \longrightarrow \cdots,
  \]
    in which the maps are obtained by applying $\Jsharp$ to standard
    foam cobordisms.
\end{theorem}

\begin{figure}[h]
    \begin{center}
        \includegraphics[scale=.46]{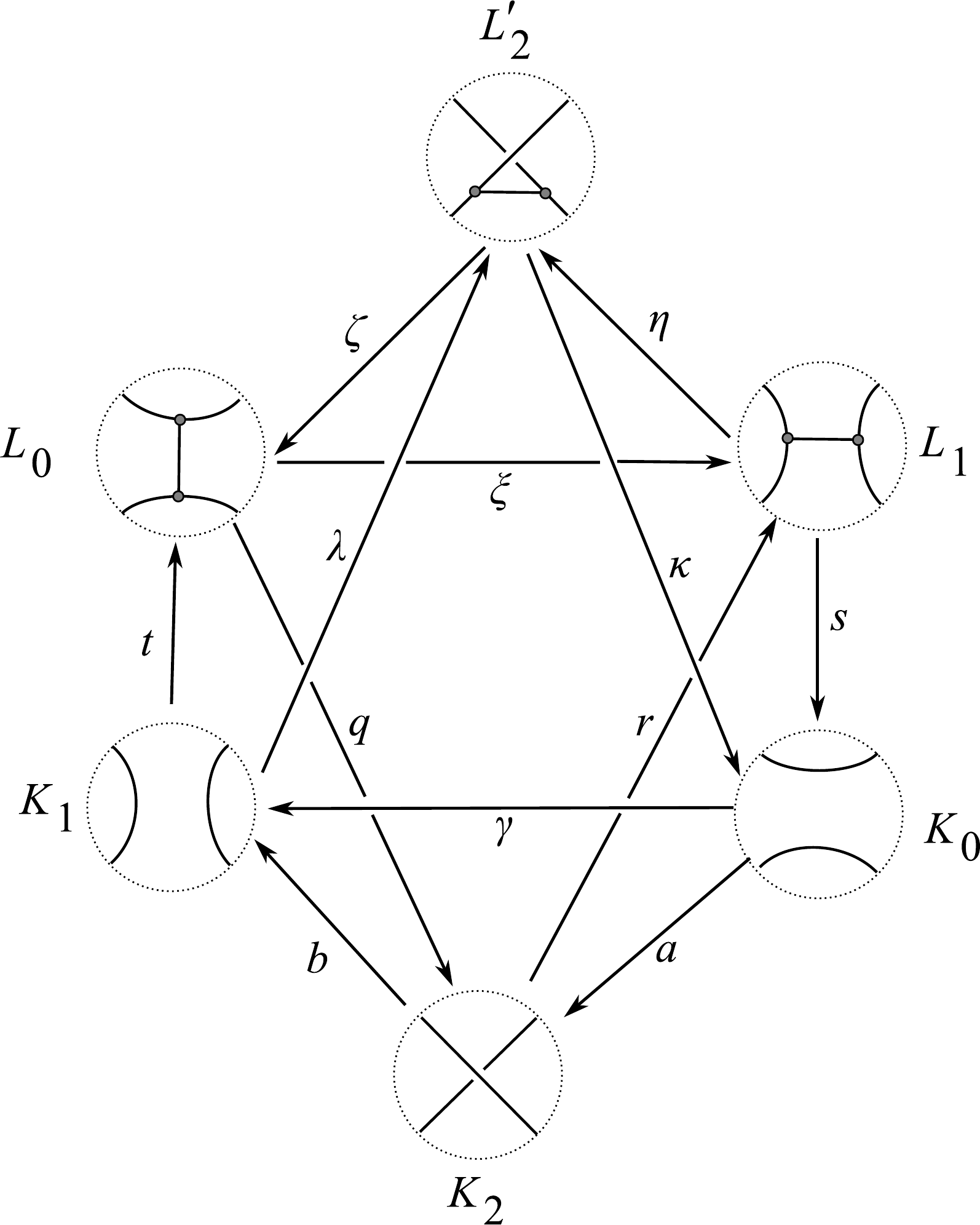}
    \end{center}
    \caption{\label{fig:J-octahedron}
    An octahedral diagram whose faces are four exact triangles and
    four commutative triangles.}
\end{figure}

The four exact triangles contained in the two theorems above can be
arranged as four of the triangular faces in an octahedral diagram, a
particular realization of the diagram for the octahedral axiom for a
triangulated category. In the diagram, Figure~\ref{fig:J-octahedron},
the top vertex $L'_{2}$ has a crossing of a different sign from the
picture of $L_{2}$. The exact triangles in this octahedron are
\[
\begin{aligned}
            \cdots &\longrightarrow \Jsharp(K_{2}) \longrightarrow
                 \Jsharp(K_{1}) \longrightarrow
                  \Jsharp(L_{0}) \longrightarrow
                    \Jsharp(K_{2}) \longrightarrow \cdots \\
            \cdots &\longrightarrow \Jsharp(K_{2}) \longrightarrow
                 \Jsharp(L_{1}) \longrightarrow
                  \Jsharp(K_{0}) \longrightarrow
                    \Jsharp(K_{2}) \longrightarrow \cdots \\
            \cdots &\longrightarrow \Jsharp(K_{0}) \longrightarrow
                 \Jsharp(K_{1}) \longrightarrow
                  \Jsharp(L'_{2}) \longrightarrow
                    \Jsharp(K_{0}) \longrightarrow \cdots \\
            \cdots &\longrightarrow \Jsharp(L_{0}) \longrightarrow
                 \Jsharp(L_{1}) \longrightarrow
                  \Jsharp(L'_{2}) \longrightarrow
                    \Jsharp(L_{0}) \longrightarrow \cdots .
\end{aligned}
\]
The last two are duals of exact triangles in
Theorems~\ref{thm:intro-one-bar} and~\ref{thm:intro-three-bar}
respectively. The other four faces of octahedron become commutative
diagrams of $\F$-vector spaces on applying $\Jsharp$. For example, the
triangle
\[
  \xymatrix{  \Jsharp(K_{1})  & & \Jsharp(K_{0}) \ar[dl] \ar[ll]\\
        & \Jsharp(K_{2}) \ar[ul]&}               
\]
 is a commutative diagram.
Finally, the two
different composites from  $K_{2}$ to $L'_{2}$,
\[
\begin{aligned}
    \mathfig{0.079}{figures/Graph-X} & \longrightarrow \mathfig{0.079}{figures/Graph-H}
    \longrightarrow \mathfig{0.079}{figures/Graph-Xbardag} \\
  \mathfig{0.079}{figures/Graph-X} & \longrightarrow \mathfig{0.079}{figures/Graph-Res1}
    \longrightarrow \mathfig{0.079}{figures/Graph-Xbardag},
\end{aligned}
\]
give the same map $\Jsharp(K_{2}) \to \Jsharp(L'_{2})$, with a similar
(and equivalent)
statement about the two composites from $L'_{2}$ to $K_{2}$. 

Fuller versions of these results are stated in
section~\ref{sec:statements}, where we also broaden the scope of the
theorems a little by discussing webs embedded in arbitrary oriented
$3$-manifolds, rather than in $\R^{3}$. 

\subparagraph{Acknowledgement.} The authors are very grateful for the support of the Radcliffe
Institute for Advanced Study, which provided them with the opportunity
to pursue this project together as Fellows of the Institute during the
academic year 2013--2014.

\section{Review of $\SO(3)$ instanton homology}
\label{sec:recap}

We briefly recall some of the constructions which are described more
fully in \cite{KM-jsharp}. If $Z$ is an $n$-dimensional orbifold, and
$z\in Z$, then we write $H_{z}$ for the local stabilizer group at
$z$. All our orbifolds will be orientable, so $H_{z}$ is a subgroup of
$\SO(n)$ acting effectively on $\R^{n}$.

\paragraph{Bifolds.} For $m\le n$, let $H_{m}\subset \SO(m)$ be the
elementary abelian $2$-group of order $2^{m-1}$ consisting of diagonal
matrices of determinant $+1$ whose diagonal entries are $\pm
1$. Regard $H_{m}$ also as a subgroup of $\SO(n)$ for $m\le n$. We
call $Z$ an $n$-dimensional \emph{bifold} if its local stabilizer
groups $H_{z}\subset \SO(n)$ are conjugate to $H_{m}$ for some $m\le
n$. All our bifolds will be equipped with Riemannian metrics, in the
orbifold sense.

\paragraph{Webs and foams.} The underlying topological space of a
bifold is a manifold $X$, and the set of points with non-trivial local
stabilizer is a codimension-$2$ subcomplex of $X$. In the case of
dimension $2$, this subcomplex is a set of points. In the case $n=3$,
it is a trivalent graph, which we refer to as a web.

In the case $n=4$, the points with $H_{z}\ne 1$ form a $2$-complex which
we call a foam. A foam can have tetrahedral points, where the local
stabilizer is $H_{4}$. The set of points with $H_{z}\cong H_{3}$ is a
union of arcs and circles: these are the seams, which together with
the tetrahedral points comprise a $4$-valent graph. The remainder of the
foam is a $2$-manifold whose components are the \emph{faces}.

A pair $(Y,K)$ consisting of a smooth $3$-manifold and smoothly embedded
web can be used to construct a corresponding bifold $\check Y$. The
same is true for a pair $(X,\Phi)$ consisting of a $4$-manifold and an
embedded foam: we may write the corresponding bifold as $\check X$. 
There is a cobordism 
category in which the objects are closed,
oriented $3$-dimensional bifolds $\check Y$ with bifold metrics, and in which the
morphisms are isomorphism classes of oriented $4$-dimensional bifolds
$\check X$
with boundary. Equivalently, we have a category in which the objects
are $3$-manifolds $(Y,K)$ with embedded webs, and the morphisms are
$4$-manifolds $(X,\Phi)$
with embedded foams.

\paragraph{Bifold connections.} By a \emph{bifold connection} over a
bifold $\check X$, we mean an $\SO(3)$ orbifold vector bundle
$E\to\check X$ equipped with an orbifold $\SO(3)$ connection $A$,
subject to the constraint that at each point $x$ where $H_{x}$ has
order $2$, the local action of $H_{x}$ on the $\SO(3)$ fiber is
non-trivial. This condition determines the local model uniquely at
other orbifold points. In particular, if $\check X$ is $4$-dimensional
and $x$ belongs to a seam of the corresponding foam, so that $H_{x}$
is the Klein $4$-group, then the representation of $H_{x}$ on the
$\SO(3)$ fiber is the inclusion of the standard Klein $4$-group
$V\subset \SO(3)$.

\paragraph{Marking data.} Bifold connections may have non-trivial
automorphisms. For example, if the monodromy group of the connection
is the $4$-group $V$, then the automorphism group is also $V$. In order
to have objects without automorphisms, we introduce \emph{marked}
bifold connections.

By \emph{marking data} $\mu$ on a bifold $\check X$, we mean a pair
$(U_{\mu}, E_{\mu})$ consisting of an open set $U_{\mu}$ and an
$\SO(3)$ bundle $E_{\mu} \to U_{\mu}\cap \check X^{o}$ (where $\check
X^{o}$ is the locus of non-orbifold points).  A \emph{marked} bifold
connection is a bifold connection $(E,A)$ on $\check X$ together with
a choice of an equivalence class of an isomorphism $\sigma$ from
$E|_{U_{\mu}\cap \check X^{o}}$ to $E_{\mu}|_{U_{\mu}\cap\check
  X^{o}}$. Two isomorphism $\sigma_{1}$ and $\sigma_{2}$ are
equivalent if $\sigma_{1}\comp\sigma_{2}^{-1} : E_{\mu}\to E_{\mu}$
lifts to the determinant-$1$ gauge group, i.e.~a section of the
associated bundle with fiber $\SU(2)$. The marking data is
\emph{strong} if the automorphism group of every $\mu$-marked bifold
connection is trivial. In dimension $3$, a sufficient condition for
$\mu$ to be strong is that $U_{\mu}$ contains a point $x$ with
$H_{x}=V$ (a vertex of the corresponding web), or that $U_{\mu}\check
X^{o}$ contains a torus on which $w_{2}(E_{\mu})$ is non-zero.

\paragraph{Instanton homology.} Let $\check Y$ be a closed, connected,
oriented $3$-dimensional bifold with strong marking data $\mu$. The set
of isomorphism classes of $\mu$-marked bifold connections of Sobolev
class $L^{2}_{k}$, for large enough $k$, is parametrized by a Hilbert
manifold $\bonf_{k}(\check Y;\mu)$. Using the perturbed Chern-Simons
functional, one constructs a Morse complex, whose homology we call the
$\SO(3)$ instanton homology. It is defined with coefficients
$\F=\Z/2$. We use the notation $J(\check Y;\mu)$. If $(Y,K)$ is a pair
consisting of a $3$-manifold and an embedded web, we similarly write
$J(Y,K;\mu)$.

Let $\check X$ be an oriented bifold cobordism from $\check{Y_{1}}$ to
$\check Y_{2}$, let $\nu$ be marking data for $\check X$, and let
$\mu_{i}$ be the restriction of $\nu$ to $\check Y_{i}$. If $\mu_{i}$
is strong, for $i=1, 2$, then $(\check X, \nu)$ gives rise to a linear
map
\[
      J(\check X;  \nu) : J(\check Y_{1};\mu_{1}) \to J(\check Y_{2};\mu_{2}).
\]
In general, the map which $J$ assigns to composite cobordism may not
be the composite map. However, the composition law does hold if the
marking data $\nu$ on the two cobordisms satisfies an extra
condition. In this paper, our cobordisms will always contain product
cobordisms in the neighborhoods of the marking data, and $\nu$ will
always be a product $[0,1]\times \mu_{1}$. This restriction is
sufficient to ensure that the composition law holds.

\paragraph{The construction of $\Jsharp$.} Let $K$ be a compact web in
$\R^{3}$. From $K$, we form a new web $K^{\sharp}\subset S^{3}$ as the
disjoint union of $K$ and a Hopf link $H$ contained in a ball near the
point at infinity. As marking data for $(S^{3}, K^{\sharp})$, we take
$U_{\mu}$ to be the ball containing $H$, disjoint from $K$, and we
take $E_{\mu}$ to have $w_{2}\ne 0$ on the torus which separates the
two components of the Hopf link. This marking data is strong. We
define
\[
\Jsharp(K) = J(S^{3}, K^{\sharp};\mu). 
\]
Given a foam cobordism $\Phi\subset [0,1]\times\R^{3}$ from $K_{1}$ to
$K_{2}$, we similarly construct a new foam $\Phi^{\sharp}$ as
$\Phi\cup ([0,1]\times H)$, with marking data $\nu=[0,1]\times\mu$. In
this way, $\Phi$ gives rise to a linear map
\[
   \Jsharp(\Phi) : \Jsharp(K_{1}) \to \Jsharp(K_{2}).
\]
In this way we obtain a functor $\Jsharp$ with values in the category
of $\F$-vector spaces, from a category whose objects are webs in
$\R^{3}$ and whose morphisms are isotopy classes of foams with
boundary in intervals $[a,b]\times \R^{3}$.

\section{Statement of the results}
\label{sec:statements}

We state now the version of the Theorem~\ref{thm:intro-three-bar}
that we shall prove.  Instead of $\R^{3}$, we consider a closed,
oriented $3$-manifold $Y$, and three webs $L_{2}$, $L_{1}$, $L_{0}$ in $Y$ which are
identical outside a standard ball $B\subset Y$. As in
Theorem~\ref{thm:intro-three-bar}, we
suppose that, inside the ball $B$, they look as shown in
\eqref{eq:intro-L-pics}. 
As with other variants of Floer's exact triangle, there is more symmetry
between the three pictures than immediately meets the eye. The same
pictures are drawn from a different point of view in the bottom row of
Figure~\ref{fig:J-skein-three-bars}, to exhibit the cyclic symmetry
between the three.
 We write $K_{i}$ (as in the introduction) for the web
obtained from $L_{i}$ by forgetting the two vertices inside the ball
and deleting the edge joining them. Similar picture of the these webs are shown in
Figure~\ref{fig:J-skein}. 
Let $\mu = (U_{\mu},
E_{\mu})$ be strong marking data with $U_{\mu}$ disjoint from $B$. 
We may regard $\mu$ as marking data for all three of the pairs $(Y,
L_{i})$ and all three of the pairs $(Y, K_{i})$.

\begin{figure}[h]
    \begin{center}
        \includegraphics[scale=.5]{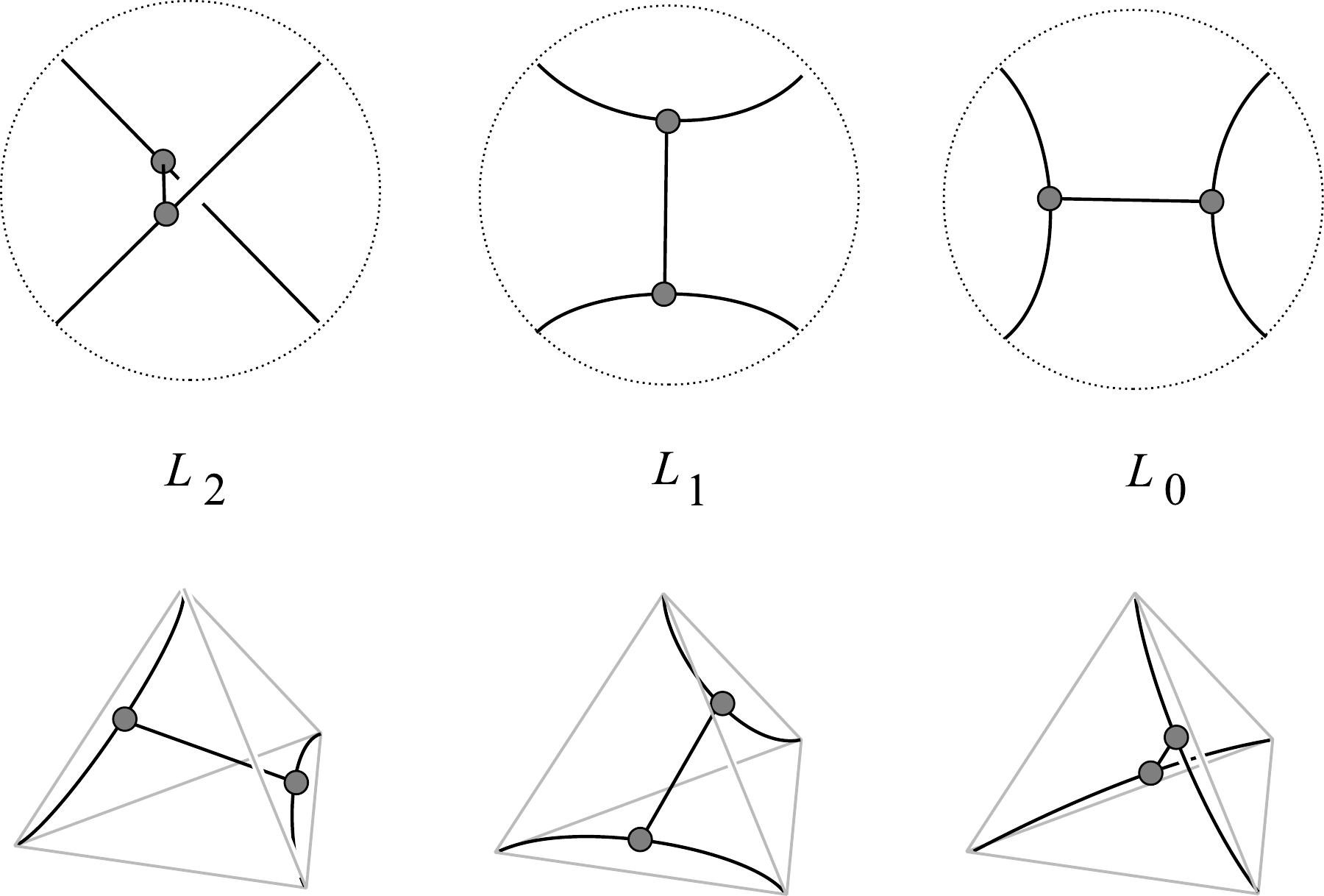}
    \end{center}
    \caption{\label{fig:J-skein-three-bars}
   The three webs $L_{i}$, from two different points of view.}
\end{figure}

\begin{figure}[h]
    \begin{center}
        \includegraphics[scale=.5]{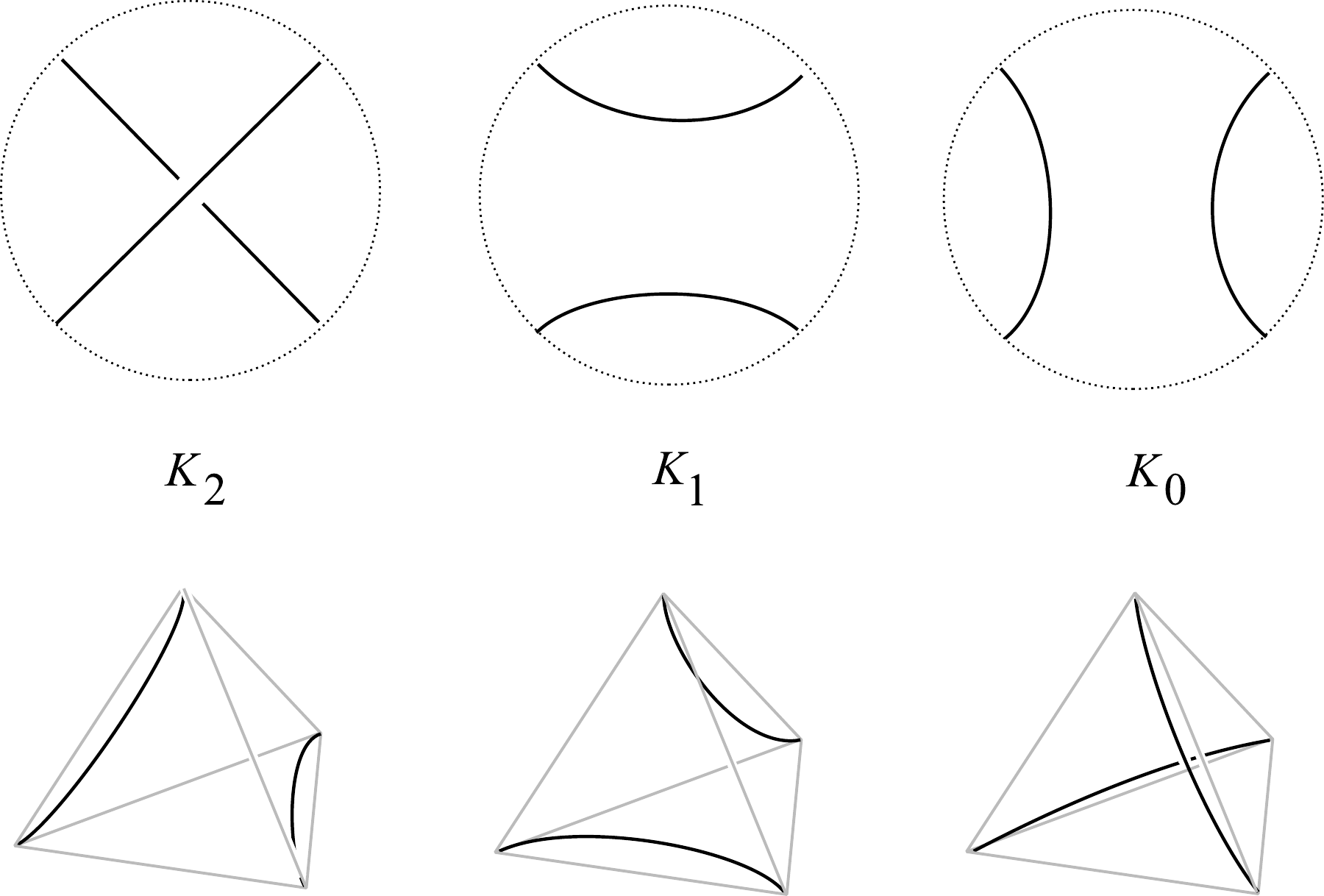}
    \end{center}
    \caption{\label{fig:J-skein}
    The three webs in $Y$ obtained from the $L_{i}$ by removing an edge.}
\end{figure}

For each $i$, there is a standard cobordism from $K_{i+1}$ to $K_{i}$
given by a foam $\Sigma(K_{i+1}, K_{i})$ in $I \times Y$.  (The index
$i$ is to be interpreted cyclically.) The cobordism in each case is
the addition of a standard $1$-handle.  There are also standard
cobordisms to and from the $L_{i}$, which we write as $\Sigma(L_{i+1},
K_{i})$, $\Sigma(K_{i+1}, L_{i})$, and $\Sigma(L_{i+1}, L_{i})$. These
are all obtained from $\Sigma(K_{i+1},K_{i})$ by adding one or two
disks.  Pictures of $\Sigma(L_{1}, K_{0})$ and $\Sigma(L_{1}, L_{0})$
are given in Figure~\ref{fig:cobordism}. The latter foam has a single
tetrahedral point. In the picture of the cobordism from $L_{1}$ to
$L_{0}$, we have labeled as $\delta_{1}$  and $\delta_{0}$ the edges
of the webs $L_{1}$ and $L_{0}$ which are contained in the interior of
the ball. These edges appear on the boundary of disks $\Delta^{+}_{1}$
and $\Delta^{-}_{0}$ in the foam $\Sigma(L_{1}, L_{0})$. The
tetrahedral point is the unique intersection point
$\Delta^{+}_{1}\cap\Delta^{-}_{0}$.

\begin{figure}[h]
    \begin{center}
        \includegraphics[scale=.35]{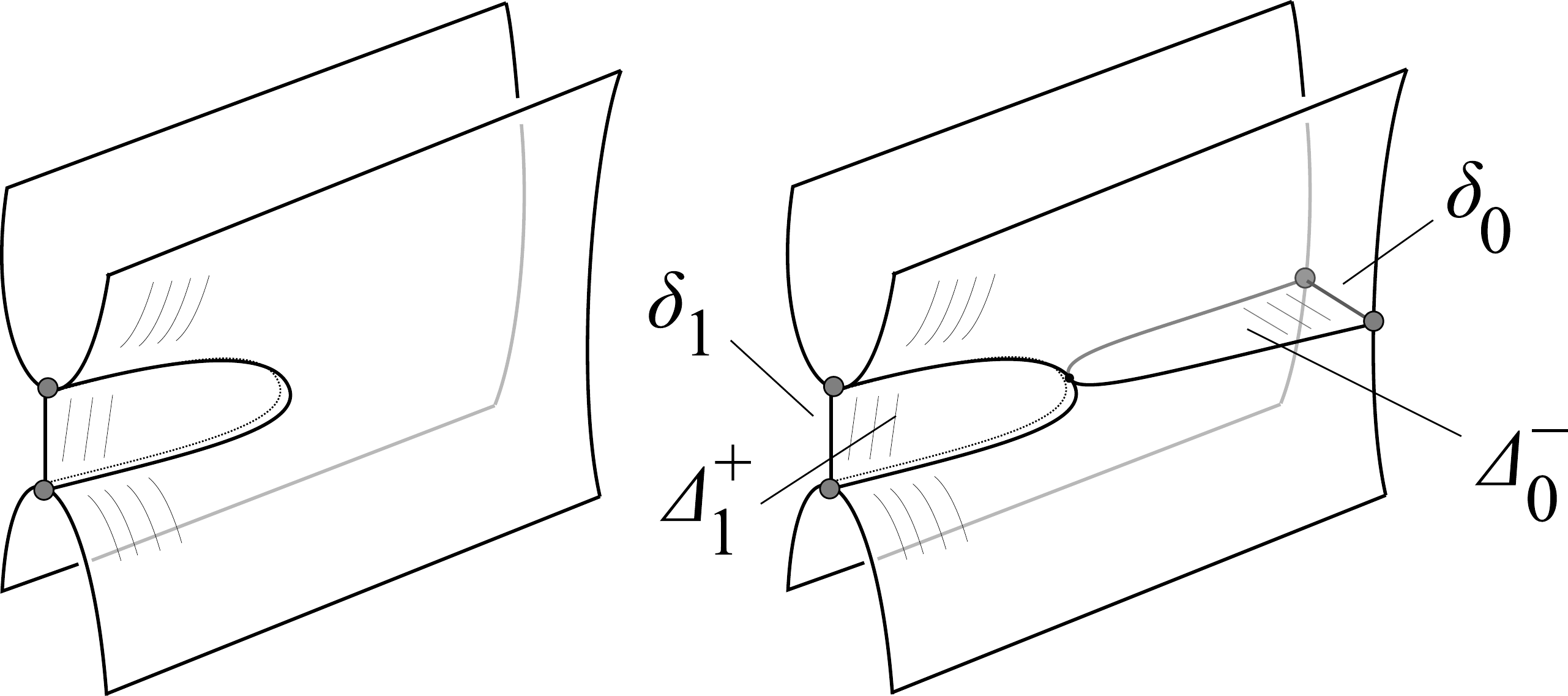}
    \end{center}
    \caption{\label{fig:cobordism}
    The cobordisms from $L_{1}$ to $K_{0}$ (left) and from $L_{1}$ to
    $L_{0}$ (right).}
\end{figure}

The standard cobordisms give maps such as
\[
                  J(I\times Y, \Sigma(L_{i+1}, L_{i}) ; \nu)
    : J(Y, L_{i+1}; \mu) \to  J(Y, L_{i}; \mu)
\]
where the marking data $\nu$ is the product $I\times \mu$. Thus we
have a sequence of maps with period $3$,
\begin{equation}\label{eq:J-exact-triangle}
        \dots \to J(Y, L_{2}; \mu) \to  J(Y, L_{1}; \mu) \to  J(Y,
        L_{0}; \mu) 
   \to  J(Y, L_{2}; \mu) \to  \cdots
\end{equation}

\begin{theorem}
    \label{thm:exact-triangle}
The above sequence is exact.
\end{theorem}

As a special case, we can consider webs in $\R^{3}$ and apply the
functor $\Jsharp$, in which case we deduce the version in the
introduction, Theorem~\ref{thm:intro-three-bar}.

There is a similar generalization of Theorem~\ref{thm:intro-one-bar} in
the  setting of foams in a $3$-manifold $Y$ with strong marking,
whose statement is easily formulated. 
It will turn out that there
is an argument that allows
Theorem~\ref{thm:intro-one-bar} to be deduced from
Theorem~\ref{thm:intro-three-bar} (or, in the more general form, from
Theorem~\ref{thm:exact-triangle}). We will therefore focus on
Theorem~\ref{thm:exact-triangle} to begin with.

\section{Calculations for some connected sums}
\label{sec:connect-sums}

The quotient of $-\CP^{2}$ by the action of complex conjugation,
$[z_{1}, z_{2}, z_{3}]\mapsto [\bar z_{1}, \bar z_{2},\bar z_{3}]$, is
an orbifold $S^{4}$ containing as branch locus the image $R\subset
S^{4}$ of $\RP^{2}\subset -\CP^{2}$. If $L$ is a complex line in
$-\CP^{2}$ defined by real linear equation in the homogeneous
coordinates, then the image of $L$ in the quotient is a disk $D$ whose
boundary is a real projective line in the branch locus $R$. Given $n$
such lines in $-\CP^{2}$, say
\[
        L_{1} ,\dots, L_{n},
\]
we obtain $n$ disks $D_{1} ,\dots, D_{n}$ in $S^{4}$ whose interiors
are disjoint and whose boundaries meet in pairs at points of $R$. The
union
\begin{equation}\label{eq:Psi-foams}
         \Psi_{n} = R \cup D_{1} \cup \cdots \cup D_{n}
\end{equation}
is a foam in $S^{4}$. This description does not specify the topology
of $\Psi_{n}$
uniquely when $n$ is large, because there are different combinatorial
configurations of real projective lines. 
The cases of most interest to us are $\Psi_{0}$
(which is just the real projective plane $R$ in $S^{4}$, with $R\cdot
R=2$), and the foams $\Psi_{1}$, $\Psi_{2}$ and $\Psi_{3}$. The foam
$\Psi_{2}$ has  a
single tetrahedral point where $\partial D_{1}$ meets $\partial D_{2}$
in $R$, while $\Psi_{3}$ has three tetrahedral points.

\begin{lemma}
    The formal dimension of the moduli space of anti-self-dual bifold
    connections of action $\kappa$ on $(S^{4}, \Psi_{n})$ is given by
    \[
                8 \kappa - (1 - n/2)^{2}.
    \]
     In particular, we have
    \begin{enumerate}
    \item $8\kappa - 1$ for $\Psi_{0}$;
    \item $8\kappa - 1/4$ for $\Psi_{1}$;
    \item $8\kappa $ for $\Psi_{2}$;
    \item $8\kappa - 1/4$ for $\Psi_{3}$;
    \end{enumerate}
\end{lemma}

\begin{proof}
    The dimension formula in general is given in
    \cite[Proposition~\ref{prop:dimension}]{KM-jsharp}, and for foams
    in $S^{4}$ it reads
    \[
            8 \kappa + \chi(\Psi_{n}) + \frac{1}{2}( \Psi_{n} \cdot
            \Psi_{n} ) - \tau(\Psi_{n}) - 3,
     \]
    where the self-intersection number $\Psi_{n} \cdot \Psi_{n}$ is
    computed face-by-face using the framing of the double-cover of the
    boundary obtained from the seam \cite{KM-jsharp}. For $\Psi_{n}$,
    there is a contribution of $+2$ from $R\cdot R$ and $-1/2$ from
    each disk $D_{i}$, so $\Psi_{n} \cdot \Psi_{n} = 2 - n/2$. The
    term $\tau$ is the number of tetrahedral points, which is
    $n(n-1)/2$, and the Euler number $\chi$ is $1+n$. The formula in
    the lemma follows.
\end{proof}

For each of the cases in the previous lemma, we can now consider
the non-empty moduli spaces of smallest possible action. For
appropriate choices of metrics, we can identify these completely.

\begin{lemma}
    On the bifolds corresponding to $(S^{4}, \Psi_{n})$, the smallest-action
    non-empty moduli spaces of anti-self-dual bifold connections  are
    as follows, for $n\le 2$.
    \begin{enumerate}
    \item For $n=0$ or $n=2$, there is a unique solution with $\kappa=0$: 
        a flat connection whose holonomy group has order $2$ for $n=0$
        and is the Klein $4$-group, $V_{4}$, for $n=2$. The
        automorphism group of the connection is $O(2)$ (respectively, $V_{4}$),
         and it is an
        unobstructed solution in a moduli space of formal dimension
        $-1$ (respectively, dimension $0$).
    \item For $n=1$ and $n=3$,
        the smallest non-empty moduli spaces have
        $\kappa=1/32$ and formal dimension $0$. In both cases,  with
        suitable choices of bifold metrics, 
        the moduli space consists of a unique unobstructed
        solution with holonomy group $O(2)$.
    \end{enumerate}
\end{lemma}

\begin{proof}
    The complement $S^{4}\sminus R$ deformation-retracts onto another copy of
    $\RP^{2}$ in $S^{4}$, which we call $R'$ (the image of $R$ under
    the antipodal map, in a standard construction of $R$). The
    interiors of the disks
    $D_{i}$ are the fibers of this retraction. So $S^{4}\sminus
    \Psi_{n}$ has the homotopy type of $R'$  with $n$ punctures. In
    particular, the fundamental group is $\Z/2$ for $n=0$ and $\Z *
    \Z$ for $n=2$. For $i=0$ and $2$, the smallest possible action is
    clearly $\kappa=0$, and we are therefore looking at
    representations of $\Z/2$ or $\Z * \Z$ in $\SO(3)$ sending the
    standard generators to involutions. In the second case, the two
    involutions must be distinct and commuting because of the presence
    of the tetrahedral point where the disks meet. So the flat
    connections are as described in the lemma. For these bifold
    connections $A$, we can read off $H^{0}_{A}$ and $H^{1}_{A}$ in
    the deformation complex by
    elementary means, and conclude that $H^{2}_{A}=0$ from the
    dimension formula.

    For the case $n=1$, the corresponding bifold admits a
    double-cover, branched along $R$, for which the total space is
    $-\CP^{2}$ containing a complex line $L$ as orbifold locus with
    cone-angle $\pi$. According to
    \cite{Abreu, Bryant},
    there exists a
    conformally anti-self-dual bifold
    metric on $-\CP^{2}$ with cone-angle $\pi$ along $L$ and positive
    scalar curvature. This metric
    is invariant under the action of complex conjugation, and it gives
    rise to a conformally anti-self-dual metric on the bifold $\check
    S^{4}$ corresponding to $(S^{4}, \Psi_{1})$. For such a metric,
    the obstruction space $H^{2}_{A}$ in the deformation complex of an
    anti-self-dual bifold connection $A$ is trivial \cite[Theorem
    6.1]{AHS}, so the moduli space is zero-dimensional and consists of
    finitely many points, all of which are unobstructed
    solutions. If $A$ is such a bifold connection, consider its
    pull-back, $\tilde A$, on the double-cover $-\CP^{2}$, regarded as a bifold with
    singular locus $L$. The action of $\tilde A$ is $1/16$, and the
    dimension formula shows that the moduli space containing $\tilde
    A$ on this bifold has formal dimension $-1$. Since it is
    unobstructed, the solution must be reducible, and must therefore
    be an $\SO(2)$ connection,  with holonomy $-1$ around the link of $L\subset
    -\CP^{2}$. There is a unique such $\SO(2)$ solution on $-\CP^{2}$
    with the correct action, and it gives rise to a unique $O(2)$
    connection on the original bifold.

    In the case $n=3$, the bifold corresponding to $\Psi_{3}\subset
    S^{4}$ has a smooth $8$-fold cover, which is $-\CP^{2}$. The
    covering map is the quotient map for the action of the
    elementary abelian group of order $8$ acting on $-\CP^{2}$,
    generated by the action of complex conjugation and the action of
    the Klein $4$-group by projective linear transformations. We equip
    the quotient bifold with the quotient metric of the Fubini-Study
    metric, so that (as in the case $n=1$) all solutions are
    unobstructed.  A solution of action $1/32$ pulls back to a
    solution of action $1/4$ on the $8$-fold cover, which must be the unique
    instanton with holonomy group $\SO(2)$ in the $\SO(3)$ bundle with
    Pontryagin number $-1$ on $-\CP^{2}$. This descends to a unique bifold
    connection on the quotient.
\end{proof}

We use these results about small-action moduli spaces to analyze
connected sums in some particular cases. In general, given foams $\Phi
\subset X$ and $\Phi'\subset X'$ with tetrahedral points $t$, $t'$ in
each, there is a connected sum
\begin{equation}\label{eq:tet-sum}
                 (X,\Phi) \#_{t,t'} (X',\Phi')
\end{equation}
performed by removing standard neighborhoods and gluing together the
resulting foams-with-boundary. Similarly, if $s$ and $s'$ are points
on seams of $\Phi$ and $\Phi'$, there is a connected sum
\[
                (X,\Phi) \#_{s,s'} (X',\Phi'),
\]
and there is a connected sum for points $f$ and $f'$ in the interiors
of faces of the two foams:
\[
                (X,\Phi) \#_{f,f'} (X',\Phi'),
\]

Although our notation does not reflect this, the connected sum
is \emph{not unique} when we summing at a tetrahedral point or a
seam. The cause of the non-uniqueness (in the case of the tetrahedral
points, for example) is that we have to choose how to identify the
$1$-skeleta of the two tetrahedra that arise as the links of $t$ and
$t'$.

 We consider a connected sum at a
tetrahedral point in the case that $(X',\Phi')$ is either $(S^{4},
\Psi_{2})$ or $(S^{4}, \Psi_{3})$.

\begin{proposition}
 \label{prop:sum-with-Psi-2-3}
    Let $(X,\Sigma)$ be a foam cobordism with strong marking data $\nu$,
    defining a linear map $J(X,\Sigma,\nu)$. Let $t$ be a tetrahedral
    point in $\Sigma$.
    \begin{enumerate}
    \item If a new foam $\tilde\Sigma$ is constructed from $\Sigma$ as
        a connected sum \[(X,\Sigma) \#_{t, t_{2}} (S^{4},\Psi_{2}), \]
        where $t_{2}$ is the unique tetrahedral point in $\Psi_{2}$,
        then the new linear map  $J(X,\tilde \Sigma,\nu)$ is equal to
        the old one.
    \item If a new foam $\tilde\Sigma$ is constructed from $\Sigma$ as
        a connected sum \[ (X,\Sigma) \#_{t, t_{3}} (S^{4},\Psi_{3}),\]
        where $t_{3}$ is any of the three tetrahedral points in $\Psi_{3}$,
        then the new linear map  $J(X,\tilde \Sigma,\nu)$ is zero.   
    \end{enumerate}
\end{proposition}

\begin{proof}
    Consider a general connected sum at tetrahedral points, as in
    equation~\eqref{eq:tet-sum}. Let $A$ and $A'$ be unobstructed
    solutions on $(X,\Phi)$ and $(X',\Phi')$. Let $U_{A}$ and $U_{A'}$
    be neighborhoods of $[A]$ and $[A']$ in their respective moduli
    spaces. The limiting holonomy of
    the connections at the tetrahedral point is the Klein $4$-group $V$,
    whose commutant in  $\SO(3)$ is also $V$. So we have moduli spaces
    of solutions with framing at $t$ and $t'$ in which $[A]$ and
    $[A']$ have neighborhoods $\tilde U_{A}$, $\tilde U_{A'}$, such
    that $U_{A} = \tilde U_{A}/ V$ and $U_{A'} = \tilde
    U_{A'}/V$. Gluing theory provides a model for the moduli space on
    the connected sum with a long neck, of the form 
       \[ 
            \tilde U_{A} \times_{V} \tilde U_{A'}.
        \]
    If the action of $V$ on $\tilde U_{A}$ is free and $U_{A'}$
    consists of the single point $[A']$, then this local model is a
    finite-sheeted covering of $U_{A}$ with fiber $V/\Gamma_{A'}$,
    where $\Gamma_{A'}\subset V$ is the automorphism group of the
    solution $A'$. 

    In particular if $A'$ is the smallest energy solution on
    $(S^{4},\Phi_{2})$, then the fiber is a single point, while for
    $(S^{4}, \Phi_{3})$ the fiber is $2$ points. For the case of
    compact, zero-dimensional moduli spaces on the connected sum,
    these local models become global descriptions when the neck is
    long, and we conclude that the moduli space whose point-count
    defines the map $J(X,\Sigma, \nu)$ is unchanged in the first case
    and becomes double-covered in the second case. In the second case,
    the new map is zero because we are working with characteristic $2$.
\end{proof}

The next proposition considers similarly the results of connected sum
at seam points, where one of the summands is $\Psi_{1}$, $\Psi_{2}$ or
$\Psi_{3}$.

\begin{proposition}
 \label{prop:seam-sums}
    Let $(X,\Sigma)$ be a foam cobordism with strong marking data
    $\nu$, as in the previous proposition. Let $s$ be a 
    point in a seam of $\Sigma$. For $n=1,2,3$, let $s_{n}$ be a point
    on a seam of $\Psi_{n}$. 
   If a new foam $\tilde\Sigma_{n}$ is constructed from $\Sigma$ as
        the connected sum \[(X,\Sigma) \#_{s, s_{n}} (S^{4},\Psi_{n}), \]
        then the new linear map  $J(X,\tilde \Sigma_{n},\nu)$ is equal to
        the old one in the case $n=2$, and is zero in the case that
        $n=1$ or $n=3$.
\end{proposition}

\begin{proof}
    The proofs are standard, modeled on the proofs in the previous proposition.
\end{proof}

Next we have a proposition about connected sums at points
interior to faces of the foams.

\begin{proposition}
 \label{prop:face-sums}
    Let $(X,\Sigma)$ be a foam cobordism with strong marking data
    $\nu$, as in the previous propositions. Let $f$ be a 
    point in the interior of a face of $\Sigma$. Let $f_{n}$ be a point
    in a face of $\Psi_{n}$. 
   If a new foam $\tilde\Sigma_{n}$ is constructed from $\Sigma$ as
        the connected sum \[(X,\Sigma) \#_{f, f_{n}} (S^{4},\Psi_{n}), \]
        then the new linear map  $J(X,\tilde \Sigma_{n},\nu)$ is equal to
        the old one in the case $n=0$, and is zero in all other cases.
\end{proposition}

\begin{proof}
    Again, this is now straightforward.
\end{proof}

We consider next a different type of connect sum. Let
$\Psi_{2}^{-}$ be the mirror image of $\Psi_{2}$. It has
self-intersection number $-1$, Euler number $2$, and one tetrahedral
point. In the description of $\Psi_{2}$ in \eqref{eq:Psi-foams}, the
surface $R$ is divided into two components by the seams of the foam. Let
$f_{2}\in R$ be a point in one of those two components.

\begin{proposition}
 \label{prop:sum-with-Psi-2-minus}
    Let $(X,\Sigma)$ be a foam cobordism with strong marking data
    $\nu$. Let $f$ be a 
    point in the interior of a face of $\Sigma$. Let $f_{2} \in
    \Psi_{2}^{-}$ be as above. 
   If a new foam $\tilde\Sigma$ is constructed from $\Sigma$ as
        the connected sum \[(X,\Sigma) \#_{f, f_{2}} (S^{4},\Psi^{-}_{2}), \]
        then the new linear map  $J(X,\tilde \Sigma,\nu)$ is equal to
        the old one.
\end{proposition}

\begin{proof}
    Consider a $0$-dimensional moduli space on $(X,\tilde\Sigma)$. Let
    $\kappa$ be the action of solutions in this moduli space. If we
    stretch the neck at the connected sum, and if we obtain in the limit a
    solution on $(X,\Sigma)$ and a solution on $(S^{4},
    \Psi_{2}^{-})$, then the action of the solutions on these two
    summands must be (respectively) $\kappa$ and $0$. The
    moduli space on $(X,\Sigma)$ with action $\kappa$ is
    zero-dimensional. The moduli space on $(S^{4}, \Psi_{2}^{-})$ with
    action $0$ consists of a unique $V$-connection $A_{V}$, but the formal
    dimension of the corresponding moduli space is $-1$, because there
    is a $1$-dimensional obstruction space $H^{2}(A_{V})$ in the
    deformation complex. The gluing
    parameter is $O(2)/V = S^{1}$, so the description of the moduli
    space on $(X,\Sigma)$ is as the zero-set of a real line bundle $h$
    over an $S^{1}$ bundle over the moduli space associated to
    $(X,\Sigma)$. 
    From this description, we see that
    \[
                 J(X,\tilde\Sigma,\nu) = \epsilon J(X,\Sigma, \nu)
    \]
    where $\epsilon\in \{0,1\}$ is the evaluation of $w_{1}(h)$ on the
    $S^{1}$ fibers.

    To describe $w_{1}(h)$, let us write the $\SO(3)$ vector bundle
    $E$ for the connection $A_{V}$ as a sum of three line bundles,
    \[
             E = L_{0} \oplus L_{1} \oplus L_{2}.
    \]
   In this decomposition, let $i$, $j$, $k$ be the non-identity
   elements of $V$, chosen to have the form
\[
             i =\diag(1,-1,-1) ,\; \; j=\diag(-1,1,-1) , 
        \;  \; k=\diag(-1,-1,1).
\]
   Up to conjugacy, the monodromy of $A_{V}$ around the link of the
   disks $D_{1}, D_{2}\subset \Psi^{-}_{2}$ is $i$. The monodromies
   around the links of the two faces $R \setminus (D_{1} \cup D_{2})$
   are $j$ and $k$. The $1$-dimensional vector space $H^{2}(A_{V})$ is
   spanned a by an $E$-valued form $\omega$ whose values lie in the
   summand $L_{0}\subset E$ (as follows from the fact that the
   branched double cover of $S^{4}$ branched over $R$ has
   $b^{+}=1$). When we form connected sum at a point $f_{2}$ in a face
   of $\Psi_{2}^{-}$, the obstruction bundle $h$ will be non-trivial on
   $S^{1}$ if and
   only if the monodromy of the link at $f_{2}$ 
   acts non-trivially on the summand
   $L_{0}$ in which $\omega$ lies. Thus $h$ is non-trivial if the
   monodromy at $f_{2}$ is $j$ or $k$, but $h$ is trivial if the
   monodromy is $i$. Under the hypotheses of the proposition, $f_{2}$
   belongs to one of the faces where the monodromy if $j$ or $k$, so
   the result follows.
\end{proof}

There is a variant of Proposition~\ref{prop:sum-with-Psi-2-minus}
which we will apply in section~\ref{sec:deducing-other}. Consider
$S^{4}$ as the union of two standard balls $B^{4}_{+}\cup
B^{4}_{-}$. Let $M\subset S^{3}$ be a standard M\"obius band whose
boundary is an unknot $U$.  Let $D^{+}$ and $D^{-}$ be standard disks
in $B^{4}_{+}$ and $B^{4}_{-}$ whose boundaries are both $U$. These
two disks together with $M$ make a foam,
\[
     \Psi = D^{+} \cup D^{-} \cup M.
\]
If the half-twist in $M$ has the appropriate sign, the
self-intersection of the surface real projective plane $R = D^{+} \cup
M$ in $S^{4}$ will be $-2$, while $D^{-}\cup M$ has self-intersection
$+2$. We can realize $(S^{4},\Psi)$, if we wish, as the $V$-quotient
of $S^{2}\times S^{2}$, where one generator of $V$ interchanges the
two factors and another generator acts as a reflection on each $S^{2}$
(so that the fixed set of the second generator is $S^{1}\times S^{2}$).
Let $f_{\pm}$ be interior points in the faces $D^{\pm}$ of $\Psi$.

\begin{proposition}
 \label{prop:sum-with-double-Mobius}
    Let $(X,\Sigma)$ be a foam cobordism with strong marking data
    $\nu$. Let $f$ be a 
    point in the interior of a face of $\Sigma$.  Let $\Psi$ be the
    foam just described, and let $g \in
    \Psi$ be an interior point of one of the faces of $\Psi$.
   Let a new foam $\tilde\Sigma$ be constructed from $\Sigma$ as
        the connected sum \[(X,\Sigma) \#_{f, g} (S^{4},\Psi). \]
        Then the new linear map  $J(X,\tilde \Sigma,\nu)$ is equal to
        the old one if $g$ belongs to $D^{+}$ or to $M$. If $g$
        belongs to $D^{-}$, then $J(X,\tilde \Sigma,\nu)$ is zero.
\end{proposition}

\begin{proof}
    This is an obstructed gluing problem of just the same sort as in
    the previous proposition. Once again, the branched doubee cover of
    $S^{4}$ over the surface $R=D^{+}\cup M$ has $b^{+}=1$, and the
    calculation proceeds as before. (One can alternatively derive this
    proposition from the previous one by showing that $\Psi$ is a
    connect sum $\Psi^{-}_{2} \#_{t,t}\Psi_{2}$ at the tetrahedral points.)
\end{proof}

\section{Topology of the composite cobordisms}

As stated in the introduction, the proof of the exact triangle
(Theorem~\ref{thm:exact-triangle}) is very little different from the
proof of a corresponding result for the $\SU(2)$ instanton knot
homology $\Isharp$. The first step is to understand the topology of
the composites of the cobordisms $\Sigma(L_{i+1}, L_{i})$. We shall
abbreviate the notation for these cobordisms to just
$\Sigma_{i+1,i}$.

 Figure~\ref{fig:mobiusshadedwithfeetfoamtwodisks}
shows (schematically) 
the composite of three consecutive foam cobordisms, from $L_{3}$
to $L_{0}$. The indices are interpreted cyclically, so $L_{3}$ and
$L_{0}$ are the same web in $Y$. To explain the picture, in the foam
$\Sigma_{1,0}$ pictured in Figure~\ref{fig:cobordism}, there are two
half disks whose removal would leave a standard saddle cobordism from
$K_{1}$ to $K_{0}$. When $\Sigma_{2,1}$ and $\Sigma_{1,0}$ are
concatenated, two half-disks are joined to form a single disk
$\Delta_{1} = \Delta^{-}_{1} \cup \Delta^{+}_{1}$. The triple composite
$\Sigma_{3,2}\cup\Sigma_{2,1}\cup\Sigma_{1,0}$ contains two such disks
$\Delta_{2}$ and $\Delta_{1}$, as well as two half-disks
$\Delta_{3}^{+}$ and $\Delta_{0}^{-}$. These disks are shown shaded in
the figure, and in a somewhat schematic manner (because the foams do
not embed in $\R^{3}$). If we remove the interiors of these disks from the foam,
what remains is a plumbing of M\"obius bands, which can be interpreted
as a composite cobordism from $K_{3}$ to $K_{0}$ and which appears
also as Figure~10 of \cite{KM-unknot}. We write $\Phi_{3,0}$ for this
composite cobordism from $L_{3}$ to $L_{0}$, and $\Phi'_{3,0}$ for the
complement of the interiors of the disks, as a cobordism from $K_{3}$
to $K_{0}$.

\begin{figure}[h]
    \begin{center}
        \includegraphics[scale=.36]{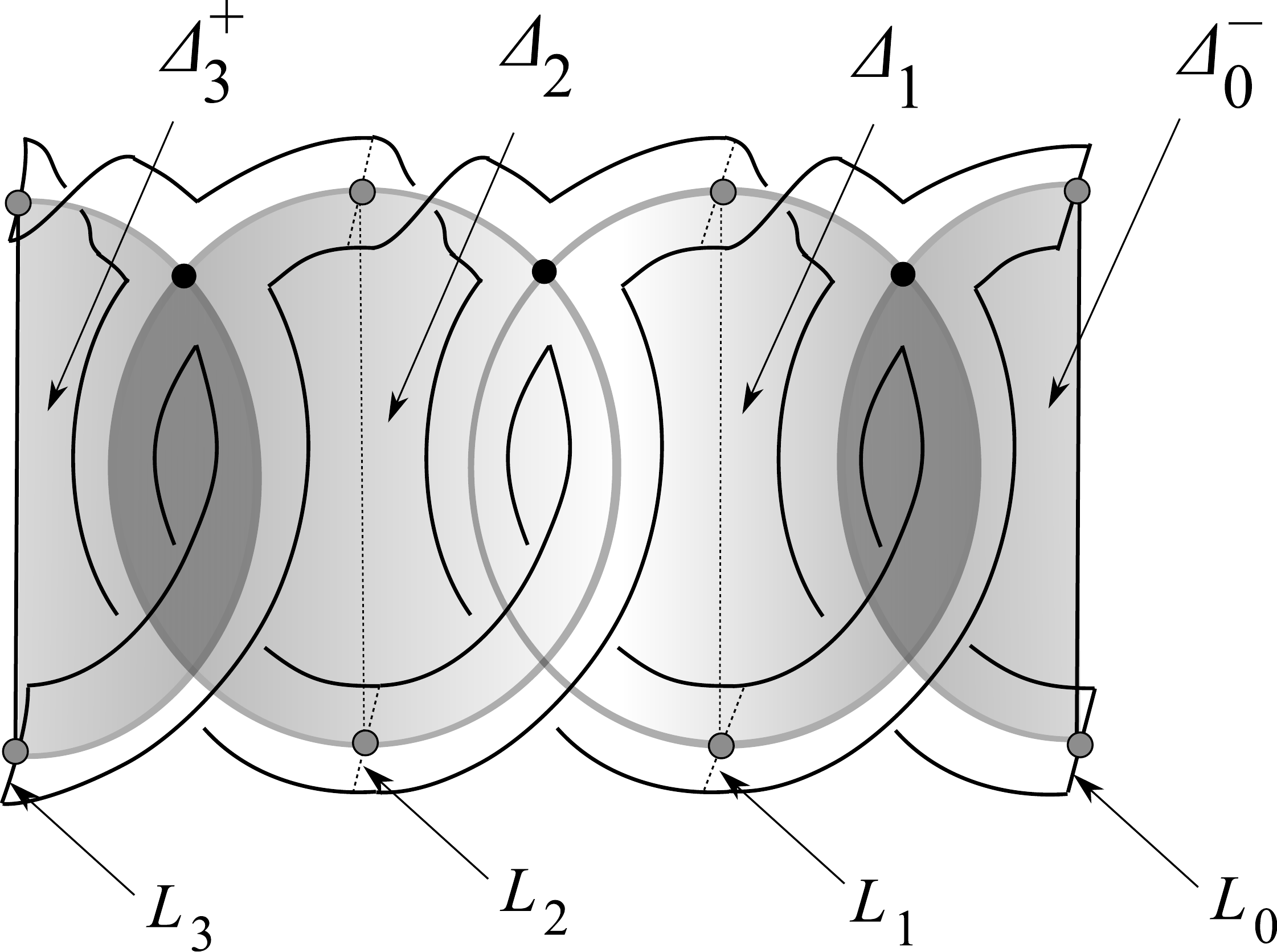}
    \end{center}
    \caption{\label{fig:mobiusshadedwithfeetfoamtwodisks}
    The composite cobordism $\Sigma_{3,2} \cup \Sigma_{2,1} \cup
    \Sigma_{1,0}$ (portrayed schematically because it is not embedded
    in $\R^{3}$). The black dots are tetrahedral points. The gray
    dots are vertices of the webs $L_{i}$.}
\end{figure}

As explained in \cite{KM-unknot}, a regular neighborhood $B_{1}$ of
$\Delta_{1}$ meets $\Phi'_{3,0}$ in a M\"obius band $M$. The same
regular neighborhood meets $\Phi_{3,0}$ in the union
\begin{equation}\label{eq:foam-union}
             M \cup \Delta_{2}^{+} \cup \Delta_{1} \cup \Delta_{0}^{-}.
\end{equation}
 This union is a foam in the $4$-ball $B_{1}$, and its boundary is the web
consisting of the boundary of $M$ and an arc on the boundary of the
half-disks $\Delta_{2}^{+}$ and $\Delta_{0}^{-}$. This web is
isomorphic to the $1$-skeleton of a tetrahedron, and the foam
\eqref{eq:foam-union} is the complement of the neighborhood of a
tetrahedral point $t_{3}$ in $\Psi_{3} = R \cup D_{1} \cup D_{2} \cup
D_{3}$. So we have an isomorphism of pairs,
\begin{equation}\label{eq:B1-pair}
    (B_{1} , \Phi_{3,0}\cap B_{1}) = (S^{4}\sminus N_{t_{3}},
    \Psi_{3}\sminus N_{t_{3}}),
\end{equation}
where $N_{t_{3}}\subset S^{4}$ is a regular neighborhood of $t_{3}$.
 In particular, we have the following counterpart to Lemma~7.2 of
\cite{KM-unknot}.

\begin{lemma}\label{lem:two-step-connect-sum}
    The composite cobordism from $(Y, L_{2})$ to $(Y,L_{0})$ formed
    from the union of the foams $\Sigma_{2,1}$ and $\Sigma_{1,0}$ has
    the form
\[
            ( I \times Y , V_{2,0} ) \#_{t,t_{3}} (S^{4},\Psi_{3}),
\]
    where $V_{2,0}$ is a foam cobordism from $L_{2}$ to $L_{0}$ with a
    single tetrahedral point $t$.
\end{lemma}

Consider next the regular neighborhood $B_{2,1}$ of the union of two disks,
$\Delta_{2}\cup\Delta_{1}$. The regular neighborhood meets
$\Phi'_{3,0}$ in a plumbing of two M\"obius bands, which is a
twice-punctured $\RP^{2}$. There is an isomorphism of pairs,
\begin{equation}\label{eq:B21-pair}
              (B_{2,1},\Phi_{3,0}\cap  B_{2,1}) = (S^{4}\sminus N_{\delta},
              \Psi_{3}\sminus N_{\delta})
\end{equation}
where $\Psi_{3}$ is as in \eqref{eq:Psi-foams} as before, and $N_{\delta}$ is
the regular neighborhood of an arc $\delta\subset D_{3}$ which joins
the two  points  $p,q \in \partial D_{3}$ in $\Psi_{3}$. The points
$p$ and $q$ lie in the interior of the two arcs which into which
$\partial D_{3}$ is divided by the two tetrahedral points.

To examine the picture of \eqref{eq:B21-pair} further, we note the web
that arises as the boundary of $\Phi_{3,0}\cap B_{2,1}$ is also the
boundary of the foam $\Omega=N_{\delta}\cap \Psi_{3}$. The foam
$\Omega$ consists of two disks which comprise $N_{\delta}\cap R$
together with the rectangle $N_{\delta}\cap D_{3}$.
\begin{figure}
    \begin{center}
        \includegraphics[scale=.30]{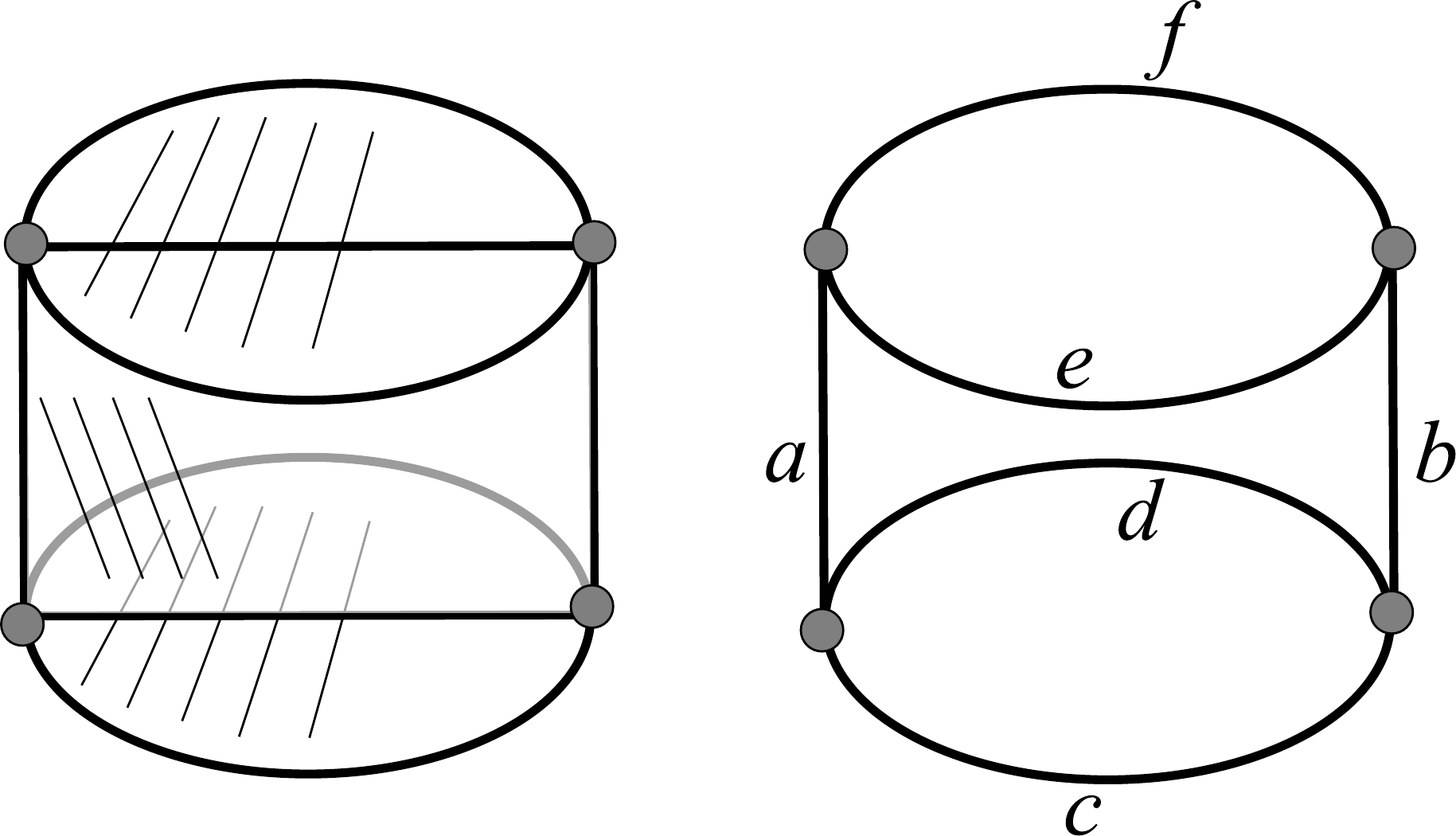}
    \end{center}
    \caption{\label{fig:Omega-and-Gamma}
    The foam $\Omega$, comprising two disks and a rectangle, and the
    web $\Gamma$ which is its boundary.}
\end{figure}
 See
Figure~\ref{fig:Omega-and-Gamma}. The foam formed from $\Phi_{3,0}$ removing
$\Phi_{3.0}\cap B_{2,1}$ and replacing it with the foam $\Omega$ is
isomorphic to the product foam $I\times L_{3} = I\times L_{0}$.
Stating this the other way round, we have the following counterpart to
Lemma~7.4 of \cite{KM-unknot}.

\begin{lemma}\label{lem:topology-of-three-step}
    The foam $\Phi_{3,0}$ in $I\times Y$ is obtained from the product
    foam $[0,1]\times L_{0}$ by removing a neighborhood $N$ of the arc
    $\{1/2\}\times \delta_{0}$ and replacing it with the foam
    $\Psi_{3}\sminus N_{\delta}$.
\end{lemma}

\section{The chain-homotopies}

In order to continue our argument, we streamline our notation. We will
then follow closely the argument in \cite{KM-unknot}. We
write simply $Y_{i}$ for the bifold corresponding to the pair
$(Y,L_{i})$. We write $X_{1,0}$ for the bifold cobordism from $Y_{1}$
to $Y_{0}$ etc., and we write $X_{3,0}$ (for example) for the
composite cobordism from $Y_{3}$ to $Y_{0}$, corresponding to the foam
$\Phi_{3,0}$ in $I\times Y$. We write $B_{2,1}$ again for the regular
neighborhood of $\Delta_{3}\cup \Delta_{1}$ which we now regard as a
bifold ball, contained in the interior of $X_{3,0}$. We also have
$B_{1}$, the bifold regular neighborhood of $\Delta_{1}$, which
arrange to be contained in the interior of $B_{2,1}$. We similarly
have $B_{2}$, the regular neighborhood of $\Delta_{2}$. We write
$S_{i}$ for the boundary of $B_{i}$ and $S_{2,1}$ for the regular
neighborhood of $B_{2,1}$.

In all, the interior of $X_{3,0}$ contains five $3$-dimensional bifolds,
\begin{equation}\label{eq:five-bifolds}
    Y_{2}, S_{2}, S_{1}, Y_{1}, S_{2,1}.
\end{equation}
Arranged cyclically in the above order, each of these five bifolds
intersects the ones before and after it, but not the other two.
We equip $X_{3,0}$ with a bifold metric which is a product in the
two-sided collar of each of the five bifolds, and arrange also that
the bifolds meet orthogonally where they intersect. Given a
$4$-dimensional bifold $Z$ with boundary, having a metric which is a
product metric on a collar of $\partial Z$, we will write $Z^{+}$ for
the complete bifold obtained by attaching cylindrical ends to the
boundary components:
\[
             Z^{+} = Z \cup [0,\infty) \times \partial Z.
\]

After choosing perturbations, we have a chain complex $(C_{i}, D_{i})$
associated to $Y_{i}$, whose homology is the instanton Floer homology
group $J(Y_{i}; \mu)$. From each $X_{i+1,i}^{+}$, we obtain chain maps,
\[
         F_{i+1,i} : C_{i+1} \to C_{i}.
\]
We will just write $F$ for $F_{i+1,i}$ and $D$ for $D_{i}$, and so
write the chain condition (mod $2$) as
\[
             FD + DF = 0.
\]

Combining Lemma~\ref{lem:two-step-connect-sum} with
Proposition~\ref{prop:sum-with-Psi-2-3}, we learn that the composite
cobordism $X_{2,0}$ gives rise to the zero map from $J(Y_{2};\mu)$ to
$J(Y_{0};\mu)$. So $F\comp F$ induces the zero map in homology. The
proof supplies an explicit chain-homotopy $J=J_{2,0}$ (or $J_{i+2,i}$
in general), so that
\begin{equation}\label{eq:first-chain-homotopy}
                F F + DJ + JD = 0.
\end{equation}
The chain-homotopy $J$ is defined by counting instantons over a
$1$-parameter family of bifold metrics $g_{t}$ on $X_{2,0}^{+}$. 
For $t=0$, the metric is the restriction of our chosen metric on
$X_{3,0}$. For $t<0$, the metric is stretched across the collar of
$Y_{1}\subset X_{2,0}$, and for $t>0$ the metric is stretched along
the collar of $S_{1}\subset X_{2,0}$. 

We have learned here that the composite of any two consecutive maps in
the sequence \eqref{eq:J-exact-triangle} is zero. To prove exactness,
following the argument of \cite[Lemma 4.2]{OS-double-covers}, it
suffices to find chain-homotopies
\[
         K_{i+3,i} : C_{i+3} \to C_{i}
\]
for all $i$, such that
\begin{equation}\label{eq:second-chain-homotopy}
           FJ + JF + DK + KD : C_{i+3} \to C_{i}
\end{equation}
is an isomorphism. 

As in \cite{KM-unknot} and \cite{KMOS}, the map $K$ is constructed as
follows. For each pair of \emph{non-intersecting} bifolds among the
five bifolds \eqref{eq:five-bifolds}, we can construct a family of
metrics on $X_{3,0}$ parametrized by the quadrant $[0,\infty) \times
[0,\infty)$, by stretching in the collars of both of the bifolds.
There are five such non-intersecting pairs, and the corresponding five
quadrants of metrics fit together to form a family of metrics
parametrized by an open disk $P$.  The map $K$ is defined by counting
points in zero-dimensional moduli spaces over the family of metrics
parametrized by $P$, on $X_{3,0}^{+}$.

The family of metrics $P$ has a natural closure $\bar P$ which is a closed pentagon
whose paramatrize certain broken metrics on $X_{3,0}^{+}$, i.e. metrics
where one (or more) of the collars has been stretched to infinity, and
we regard the limiting space has having two (or more) new cylindrical
ends. Each side of the pentagon corresponds to a family of metrics
which is broken along one of the five $3$-dimensional bifolds, and the
vertices correspond to metrics which are broken along two of them (a
pair of bifolds which do not intersect). We write
\begin{equation}\label{eq:boundary-of-P}
          \partial P
                     = Q_{Y_{2}}\cup Q_{S_{2}}\cup Q_{S_{1}}\cup
                     Q_{Y_{1}}\cup Q_{S_{2,1}}.
\end{equation}

To prove \eqref{eq:second-chain-homotopy}, one considers
one-dimensional moduli spaces on the bifold $X_{3,0}$ over the parameter space
$P$, and applies as usual the principal that a one-manifold has an
even number of ends. The compactification of such a one-dimensional
moduli space contains points of a type that did not arise in the
argument in \cite{KM-unknot}, namely those corresponding to the bubbling off of
an instanton at a
tetrahedral point, which is a codimension-$1$ phenomenon. However, as in
\cite[Section~\ref{sec:functoriality}]{KM-jsharp}, the number of
endpoints of the moduli space which are accounted for by such bubbling
is even, so there is no new contribution from these endpoints. 
We arrive at a standard formula,
\[
               DK + KD + W = 0,
\]
where $W$ is a linear  maps defined by counting the number of
endpoints of the compactified moduli space which lie over $\partial
P$. Following the description of $\partial P$ as the union of five
parts in \eqref{eq:boundary-of-P}, we write $U$ as a sum of five
corresponding terms:
\begin{equation}\label{eq:proto-second-chain}
            DK + KD + U_{Y_{2}} +  U_{S_{2}} +  U_{S_{1}} +  U_{Y_{1}}
            +  U_{S_{2,1}} = 0.
\end{equation}

In the equation \eqref{eq:proto-second-chain}, the terms $U_{Y_{2}}$
and $U_{Y_{1}}$ are respectively $JF$ and $FJ$. The terms $U_{S_{1}}$
and $U_{S_{2}}$ are both zero, because they correspond to a connect
sum decomposition at a tetrahedral point, when one of the summands is
$\Psi_{3}$. So the formula reads
\[
             DK + KD + JF + FJ = U_{S_{2,1}}
\]
and we must show that $U_{S_{2,1}}$ is chain-homotopic to the
identity.

\section{Completing the proof}

The term $U_{S_{2,1}}$ counts endpoints of the compactified moduli
space of $X_{3,0}^{+}$ that arise as limit points when the length of the collar of
$S_{2,1}$ is stretched to infinity. As in \cite{KM-unknot},
identifying the number of such endpoints is a gluing problem, for
gluing along $S_{2,1}$. 

The two orbifolds that are being glued in this
case are as follows. The first piece is the regular neighborhood 
\[
       W \supset S_{1} \cup S_{2}
\]
of
the union $S_{1} \cup S_{2}$ in $X_{3,0}^{+}$. The second piece is the
complement of $W$. If we adapt Lemma~\ref{lem:topology-of-three-step}
from the language of foams to the language of orbifolds, we obtain a
description of these two orbifolds. The complement of $W$ in
$X_{3,0}^{+}$ is isomorphic to the complement of the arc $\delta_{0}$
in the cylindrical orbifold $\R\times Y_{0}$. Meanwhile, $W$ is
isomorphic to the complement of an arc $\delta$ in the orbifold
$4$-sphere $Z_{3}$ which corresponds to the foam $\Psi_{3} \subset
S^{4}$.

We write $W^{+}$ for the cylindrical-end orbifold obtained by
attaching $\R^{+}\times S_{2,1}$ to $W$. The arc
$Q_{S_{2,1}}\subset \partial P$ parametrizes a $1$-parameter family of
metrics on $W^{+}$. The main step now is to understand the
$1$-dimensional moduli space $M_{W}$ of solutions on $W^{+}$, lying over
this $1$-parameter family of metrics, and to understand the map
$M_{W}$ to the representation variety of the end $S_{2,1}$. 

\begin{proposition}
Let $R(S_{2,1})$ denote the representation variety parametrizing flat bifold
connections on $S_{2,1}$. Let $G$ denote the
one-parameter family of metrics on $W^{+}$ corresponding the interior
of the interval $Q(S_{2,1})$.  Let $M_{W}$ denote the $1$-dimensional
moduli space of solutions on $W^{+}$ over the family of metrics $G$,
and let $r$ be the restriction map to the end:
\[
           r: M_{W} \to R(S_{2,1})
\]
Then $R(S_{2,1})$ is a closed interval, and
for generic choice of
metric perturbations, $r$ maps $M_{W}$ properly and surjectively to
the interior of the interval $R(S_{2,1})$ with degree $1$ mod $2$.
\end{proposition}

\begin{proof}
    The orbifold $S_{2,1}$ corresponds to the web $\Gamma\subset
    S^{3}$ show in Figure~\ref{fig:Omega-and-Gamma}. In an $\SO(3)$
    representation corresponding to a  flat bifold
    connection on $S_{2,1}$, the generators corresponding to the
    edges $a$ and $b$ will map to the same involution (say $i$) in $\SO(3)$. The
    generators corresponding to the remaining four edges map to
    involutions which are rotations about axes orthogonal to the axis
    of $i$. Up to conjugacy, the representation is determined by a
    single angle 
     \[
                 \theta \in [0,\pi/2]
     \]
    which is the angle between the axes of rotation corresponding to
    the edges $c$ and $e$. Thus $R(S_{2,1})$ is an interval.

    The solutions belonging to the $1$-dimensional moduli space $M_{W}$ have action
    $\kappa=1/32$. Since the smallest action that can occur at a
    bubble is $1/8$, there is no possibility of non-compactness due to
    bubbling, nor can action be lost from the cylindrical end.  The
    moduli space $M_{W}$ is therefore proper over the interior
    of $G$.

     The two limit points of the $1$-parameter family of metrics
     $G$ correspond to pulling out a
     neighborhood of either $\Delta_{1}$ or $\Delta_{2}$ from
     $W$. (See Figure~\ref{fig:mobiusshadedwithfeetfoamtwodisks}.) In
     either case, this is a connect-sum decomposition of $W^{+}$, in
     which the summand that is being pulled off is a copy of $Z_{3}$
     (the orbifold corresponding to $\Psi_{3}$) and the sum is at a
     tetrahedral point. Thus, in both cases, we see  connect-sum
     decompositions,
     \[
     \begin{aligned}
         W &= Z_{3} \#_{t,t'} W'_{1} \\
                  W &= Z_{3} \#_{t,t'} W'_{2} \\
     \end{aligned}
     \]
     corresponding to pulling out a neighborhood of $\Delta_{1}$ or
     $\Delta_{2}$ respectively. The orbifolds $W'_{1}$ and $W'_{2}$
     correspond to the foams $F_{1}$ and $F_{2}$ shown in
\begin{figure}
    \begin{center}
        \includegraphics[scale=.30]{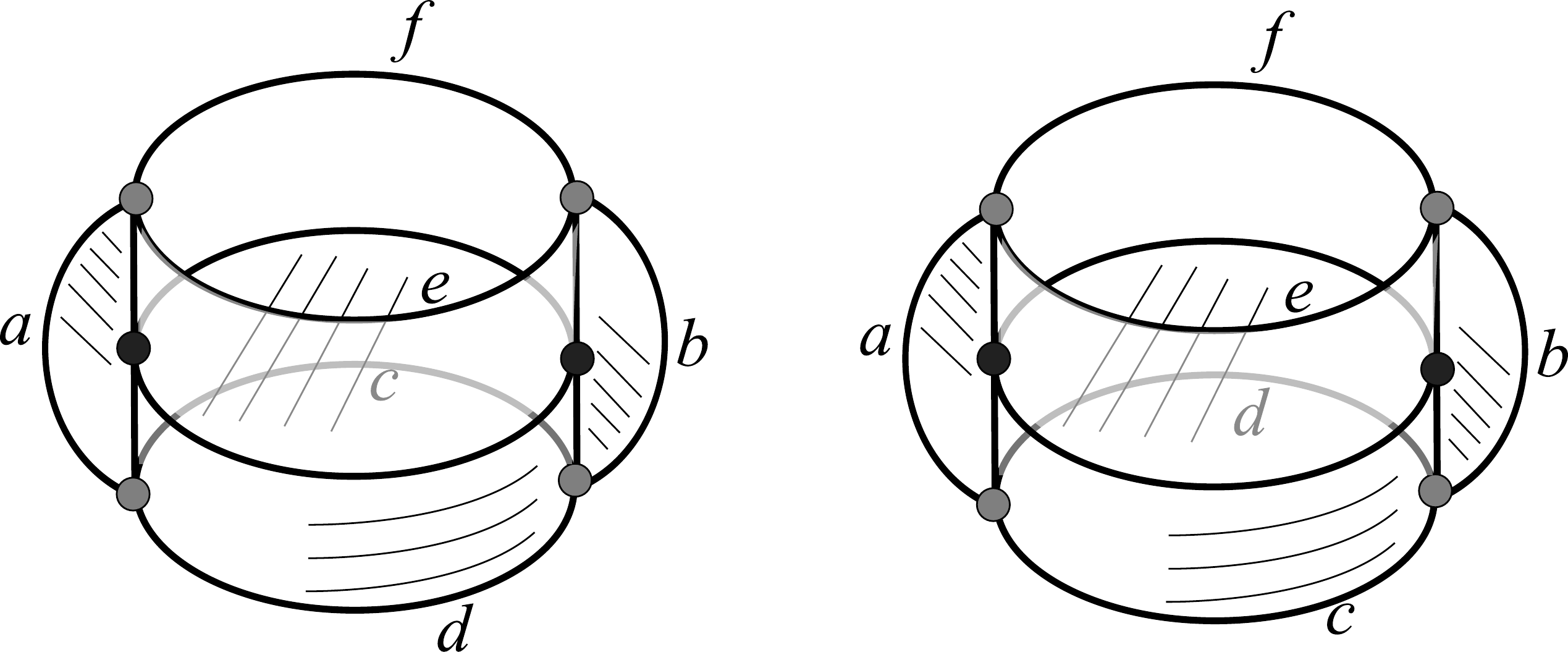}
    \end{center}
    \caption{\label{fig:foams-for-K}
    Two isomorphic foams $F_{1}$ (left) and $F_{2}$ (right),
    corresponding to the bifolds $W_{1}'$ and $W_{2}'$.
      Their boundaries are identified with $\Gamma$
    in two different ways.}
\end{figure}
     Figure~\ref{fig:foams-for-K}, regarded as a foams in $B^{4}$ with
     boundary $\Gamma$. These two foams are isomorphic, but not by an
     isomorphism which is the identity on the boundary. Thus, as
     drawn, $F_{1}$ and $F_{2}$ are the same foam, with two different
     identifications of the boundary with $\Gamma$. The
     identifications are indicated by the labeling of the edges in the
     figure.

    Since the smallest action of a moduli space on $Z_{3}$ is $1/32$,
    the limiting solution on $W'_{1}$ or $W'_{2}$ in either case must
    have action $0$, and must therefore be flat. For each of the
    (isomorphic) bifolds $W'_{1}$ and $W'_{2}$,  there is just one
    flat bifold connection up to conjugacy.
\begin{figure}
    \begin{center}
        \includegraphics[scale=.30]{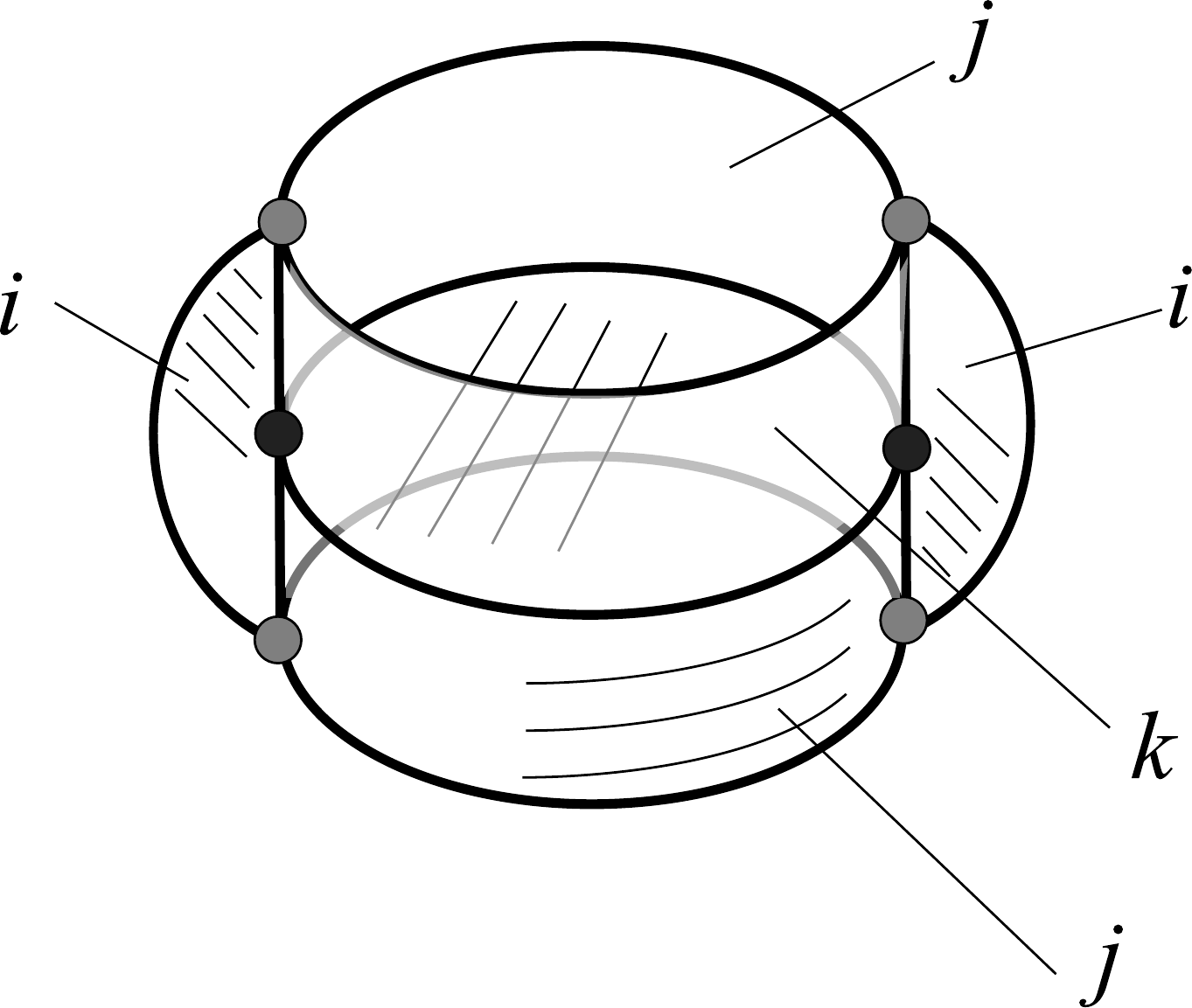}
    \end{center}
    \caption{\label{fig:V-connection}
    The $V_{4}$-connection for the bifolds $W'_{i}$.}
\end{figure} 
    This unique connection is
    a $V_{4}$-connection, in which the links of the faces are mapped to
    the non-trivial elements $\{i,j,k\}$ of $V_{4}$ as shown in
    Figure~\ref{fig:V-connection}. Inspecting
    Figure~\ref{fig:foams-for-K}, we see that, in the case of
    $W'_{1}$, the links of the edges $c$ and $e$ for $\Gamma$ are both
    mapped to $j$, while for $W'_{2}$, one is mapped to $j$ and the
    other to $k$. That is to say, the unique bifold connection on
    $W'_{1}$ (respectively, $W'_{2}$) restricts to the endpoint $0$
    (respectively $\pi/2$) in the representation variety $R(S_{2,1})
    \cong [0,\pi/2]$.  

    To summarize this, if we divide the ends of $M_{W}$ into two
    classes, according to which end of $G$ they lie over, then  the
    restriction map $r$ maps the ends of $M_{W}$ belonging to these
    two classes properly to the two different ends of the interval
    $(0,\pi/2)$. An analysis of the glueing problem for the connected
    sum shows that there is just one end in each class. Indeed, the connected
    sum is the same as the second case of
    Proposition~\ref{prop:sum-with-Psi-2-3}, but now the gluing parameter
    $V$ extends as the group of automorphisms of the bifold connection
    on $W'_{1}$ (or $W'_{2}$), so we have just one end instead of two.
   The
    proposition follows.
\end{proof}

The remainder of the proof that $U_{S_{2,1}}$ is chain-homotopic to the
identity follows the argument in \cite{KM-unknot}; and
Theorem~\ref{thm:exact-triangle} follows.

\section{Deducing the other exact triangles}
\label{sec:deducing-other}

We have now proved Theorem~\ref{thm:exact-triangle}, which becomes
Theorem~\ref{thm:intro-three-bar} when applied to webs in $\R^{3}$. We
now turn to Theorem~\ref{thm:intro-one-bar}. Because of the cyclic
symmetry, it is only necessary to treat one of the three cases, so we
take $i=1$. A proof can be constructed by following the same general
outline as the previous sections. The composite cobordisms have the
same description as shown in
Figure~\ref{fig:mobiusshadedwithfeetfoamtwodisks}, except that the
disks $\Delta_{0}^{-}$, $\Delta_{2}$ and $\Delta_{3}^{+}$ are
absent. The proof proceeds as before, except that the role of
$\Psi_{2}$ will now be played by $\Psi_{0}$ and the role of $\Psi_{3}$
by $\Psi_{1}$. We also lose some of the symmetry between the three
webs, so the proof needs to treat three cases.

Rather than repeat the details, we give here an alternative argument,
showing how to deduce Theorem~\ref{thm:intro-one-bar} from 
Theorem~\ref{thm:intro-three-bar}  by applying the basic properties of
$\Jsharp$ and the results of section~\ref{sec:connect-sums}.

  Let $L_{2}$, $L_{1}$ and $L_{0}$ be three webs which again differ
  only inside a ball, as indicated in \eqref{eq:intro-L-pics}. Let
  $\tilde L_{i}$ be obtained from $L_{i}$ by attaching an extra edge,
  as shown in the following diagram:
\begin{equation}\label{eq:L-tilde-pics}
        \tilde L_{2} = \mathfig{0.10}{figures/Graph-Xbar-extra}, \qquad
        \tilde L_{1} = \mathfig{0.10}{figures/Graph-H-extra}, \qquad
        \tilde L_{0} = \mathfig{0.10}{figures/Graph-I-extra}.
\end{equation}
In each case, the added edge is the top edge in the diagram, which we
call $e$, so
\[
            \tilde L_{i} = L_{i} \cup e.
\]
We have the standard cobordisms $\Sigma_{i,i-1}$ from $L_{i}$ to
$L_{i-1}$ as before, and these give rise to cobordisms
\[
            \tilde\Sigma_{i,i-1} = \Sigma_{i,i-1} \cup [0,1]\times e
\]
from $\tilde L_{i}$ to $\tilde
L_{i-1}$. 

Theorem~\ref{thm:intro-three-bar} tells us that we have an
exact sequence
  \[
        \cdots \longrightarrow \Jsharp(\tilde L_{2}) \longrightarrow
                 \Jsharp(\tilde L_{1}) \longrightarrow
                  \Jsharp(\tilde L_{0}) \longrightarrow
                    \Jsharp(\tilde L_{2}) \longrightarrow \cdots
  \]
where the maps are those arising from the cobordisms
$\tilde\Sigma_{i,i-1}$, or in pictures:
\begin{equation}
    \label{eq:tilde-L-sequence}
        \cdots \longrightarrow   \mathfig{0.10}{figures/Graph-Xbar-extra} \longrightarrow
                  \mathfig{0.10}{figures/Graph-H-extra}  \longrightarrow
                    \mathfig{0.10}{figures/Graph-I-extra}  \longrightarrow
                      \mathfig{0.10}{figures/Graph-Xbar-extra}  \longrightarrow \cdots
\end{equation}
where the application of $\Jsharp$ to these terms is implied.
 If we apply the ``triangle relation''
\cite[Proposition~\ref{prop:triangle-relation}]{KM-jsharp} to $\tilde
L_{0}$ we see that there is an isomorphism on $\Jsharp$ between
$L_{0}$ and $\tilde L_{0}$:
\[
    \mathfig{0.10}{figures/Graph-I}  \stackrel{\cong}{\longrightarrow}
    \mathfig{0.10}{figures/Graph-I-extra}   .
\]
From the square relation
\cite[Proposition~\ref{prop:square-relation}]{KM-jsharp}, we obtain
isomorphisms
\begin{equation}\label{eq:square-relation-1}
        \mathfig{0.10}{figures/Graph-Res1}  \oplus \mathfig{0.10}{figures/Graph-Res0}
           \stackrel{\cong}{\longrightarrow}  \mathfig{0.10}{figures/Graph-H-extra} 
\end{equation}
and
\begin{equation}\label{eq:square-relation-2}
        \mathfig{0.10}{figures/Graph-X}  \oplus \mathfig{0.10}{figures/Graph-Res0}
           \stackrel{\cong}{\longrightarrow} \mathfig{0.10}{figures/Graph-Xbar-extra} .
\end{equation}
Using these isomorphism to substitute for the terms in the exact
sequence \eqref{eq:tilde-L-sequence}, we obtain an isomorphic exact
sequence
\begin{equation}
    \label{eq:mod-sequence}
        \cdots \stackrel{\sigma_{2}}{\longrightarrow}
             \begin{bmatrix}
            \mathfig{0.08}{figures/Graph-X} \\ \oplus \\ \mathfig{0.08}{figures/Graph-Res0}
            \end{bmatrix} 
       \stackrel{\sigma_{1}}{\longrightarrow}
              \begin{bmatrix}
                   \mathfig{0.08}{figures/Graph-Res1} \\ \oplus \\
                   \mathfig{0.08}{figures/Graph-Res0}
                  \end{bmatrix} 
      \stackrel{\sigma_{0}}{\longrightarrow}
                    \mathfig{0.08}{figures/Graph-I} 
        \stackrel{\sigma_{2}}{\longrightarrow}
             \begin{bmatrix}
            \mathfig{0.08}{figures/Graph-X} \\ \oplus \\ \mathfig{0.08}{figures/Graph-Res0}
            \end{bmatrix}
    \stackrel{\sigma_{1}}{\longrightarrow} \cdots
\end{equation}
The maps $\sigma_{i}$ in this sequence are obtained from the
foam cobordisms $\tilde\Sigma_{i+1,i}$ and the isomorphisms from the
triangle and square relations. 

The claim in Theorem~\ref{thm:intro-one-bar} (for $i=0$) is that there is an exact
sequence
\begin{equation}
    \label{eq:pic-one-bar-sequence}
        \cdots \stackrel{\tau_{2}}{\longrightarrow}
            \mathfig{0.08}{figures/Graph-X}
       \stackrel{\tau_{1}}{\longrightarrow}
                   \mathfig{0.08}{figures/Graph-Res1}
      \stackrel{\tau_{0}}{\longrightarrow}
                    \mathfig{0.08}{figures/Graph-I} 
        \stackrel{\tau_{2}}{\longrightarrow}
            \mathfig{0.08}{figures/Graph-X}
    \stackrel{\tau_{1}}{\longrightarrow} \cdots
\end{equation}
where the maps $\tau_{i}$ arise from the standard cobordisms. The
exactness of this sequence will follow if we can show the following
relations between the maps:

\begin{proposition}\label{prop:same-maps}
    The maps $\sigma_{i}$ and $\tau_{i}$ in the above diagrams are
    related by
  \[
      \sigma_{1} = \begin{bmatrix} \tau_{1} & 0 \\ 0 &
          1 \end{bmatrix}
      , \qquad
      \sigma_{0} = \begin{bmatrix} \tau_{0} & 0\end{bmatrix}  
       , \qquad
              \sigma_{2} = \begin{bmatrix} \tau_{2} \\
                  0 \end{bmatrix}.
  \]
\end{proposition}

For the proof of the proposition, we need the following lemma.

\begin{lemma}
    Let $\Sigma_{2,1}$ be the standard cobordism with a single
    tetrahedral point,
   \[
                 \mathfig{0.10}{figures/Graph-Xbar}
                 \stackrel{\Sigma_{2,1}}{\longrightarrow} 
                 \mathfig{0.10}{figures/Graph-H}   .            
    \]
    Let $T_{2,1}$ be the standard cobordism from $K_{2}$ to $K_{1}$,
     \[
                    \mathfig{0.10}{figures/Graph-X}
                 \stackrel{T_{2,1}}{\longrightarrow} 
                 \mathfig{0.10}{figures/Graph-Res1} ,
     \]
     and let $\tilde T_{2,1}$ be the union of $T_{2,1}$ with a product
     $[0,1]\times f$, where $f$ is an extra edge:
      \[
                 \mathfig{0.10}{figures/Graph-Xbar-shift}
                 \stackrel{\tilde T_{2,1}}{\longrightarrow} 
                 \mathfig{0.10}{figures/Graph-H-shift}   .            
      \]
     Then $\Sigma_{2,1}$ and $\tilde T_{2,1}$ give rise to the same
     map on $\Jsharp$.
\end{lemma}

\begin{proof}
    The cobordism $\tilde T_{2,1}$ is isomorphic to a connect sum of
    $\Sigma_{2,1} \#_{t,t_{2} }\Psi_{2}$, where $t$ is the tetrahedral
    point. The result follows from Proposition~\ref{prop:sum-with-Psi-2-3}.
\end{proof}

\begin{proof}[Proof of Proposition~\ref{prop:same-maps}]
    We illustrate the arguments with one case, showing that the top
    left entry in the matrix for $\sigma_{1}$ is equal to
    $\tau_{1}$. 
\begin{figure}
    \begin{center}
        \includegraphics[scale=.50]{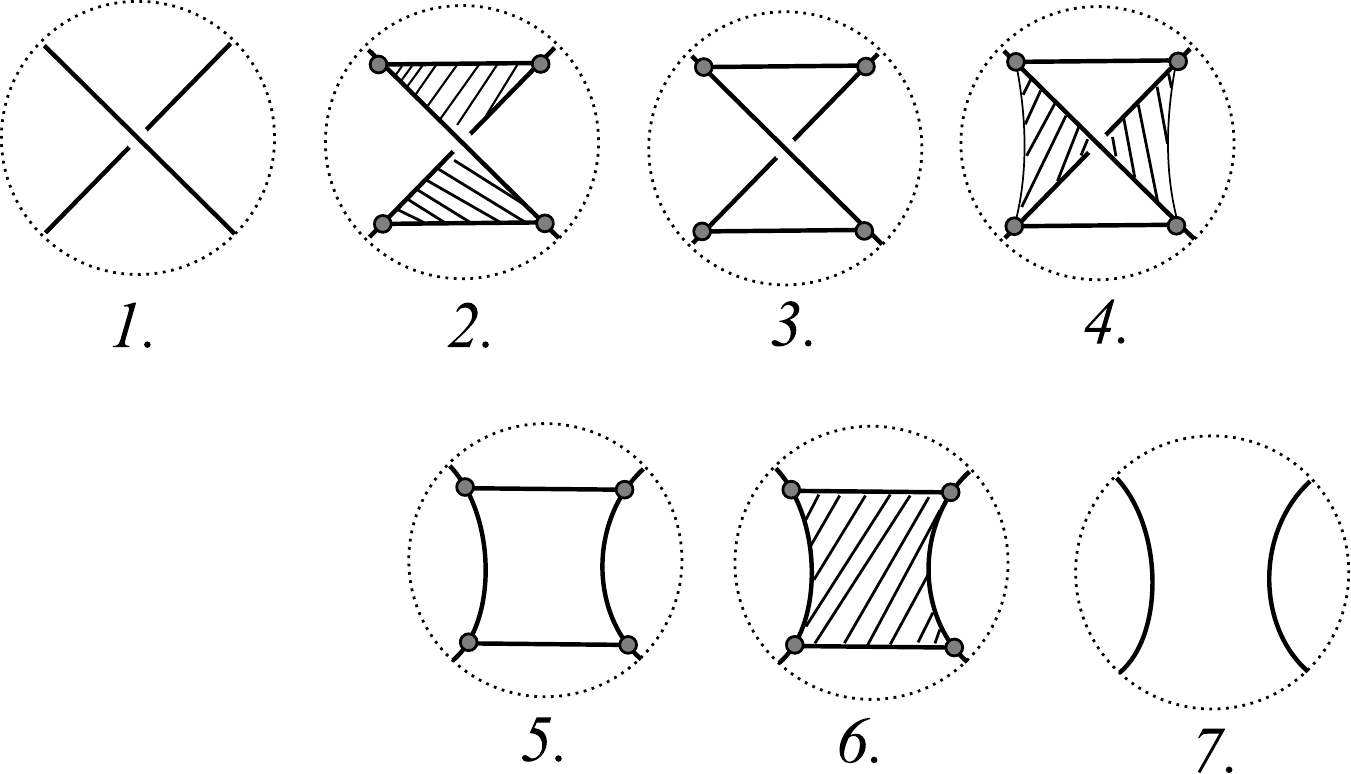}
    \end{center}
    \caption{\label{fig:sigma-is-tau}
    A movie for a foam equivalent to the top-left entry of $\sigma_{1}$.}
\end{figure}  
Consider the foam described by the movie in Figure~\ref{fig:sigma-is-tau}.    
The first
    three frames of the movie realize the first component of the
    isomorphism \eqref{eq:square-relation-2}. Frames 3--5 are the
    addition of a standard $1$-handle, which realize the same map as
    $\tilde\Sigma_{2,1}$, by the lemma above. Frames 5--7 realize the
    first component of the inverse of the isomorphism
    \eqref{eq:square-relation-1}. Taken together, the foam described
    by all seven frames gives a map equal to the top-left component of
    $\sigma_{1}$. 

    Regard the movie as defining a foam $S$ in the $4$-ball $[1,7]
    \times B^{3}$. The boundary of this foam is an unknot consisting
    of the two arcs at $t=1$, the two arcs at $t=7$, and the product
     $[1,7] \times \{\text{four points on the boundary}\}$. Let $\bar
    S$
    be closed foam in $\R^{4}$ obtained by attaching a disk  $D$ to this
    unknot, on the outside of the ball. Then, tautologically,
    \[
                      S = T_{2,1} \# \bar S,
     \]
    where $T_{2,1}$ is the standard $1$-handle cobordism from $K_{2}$ to
    $K_{1}$, and the connect sum with $\bar S$ is made at a point of
    $D\subset \bar S$. The claim is therefore that $T_{2,1}$ and $T_{2,1}\#
    \bar S$ define the same map. An examination of the movie shows
    that  \[ \bar S \cong \Psi, \] where $\Psi$ is the foam that 
    appears in Proposition~\ref{prop:sum-with-double-Mobius}, in such
    a way that the disk $D$ in $\bar S$ corresponds to the disk
    $D^{+}$ in $\Psi$. So the
    present claim follows from that proposition.
\end{proof}

\section{The octahedral diagram}

We now turn to the diagram in Figure~\ref{fig:J-octahedron}, and will
verify the properties discussed in the introduction. They are
summarized in the following theorem.

\begin{theorem}\label{thm:octahedron-statement}
In the diagram of standard cobordisms pictured in
Figure~\ref{fig:J-octahedron}, the triangles involving
\begin{enumerate}
\setlength{\itemsep}{0pt}
\item $L_{0}$, $K_{2}$, $K_{1}$,
\item $L_{1}$, $K_{0}$, $K_{2}$, 
\item $L'_{2}$, $K_{0}$, $K_{1}$, and
\item $L'_{2}$, $L_{0}$, $L_{1}$
\end{enumerate}
become exact triangles on applying $\Jsharp$. The faces
\begin{enumerate}
\setcounter{enumi}{4}
\setlength{\itemsep}{0pt}
\item \label{item:first-comm-face} 
      $K_{0}$, $K_{2}$, $K_{1}$,
\item  \label{item:second-comm-face}
      $L_{0}$, $K_{2}$, $L_{1}$,
\item  \label{item:third-comm-face}
      $K_{1}$, $L_{2}'$, $L_{0}$, and
\item  \label{item:fourth-comm-face}
      $L_{1}$, $L_{2}'$, $K_{0}$
\end{enumerate}
become commutative diagrams. And finally,
\begin{enumerate}
\setcounter{enumi}{8}
\setlength{\itemsep}{0pt}
\item the composites $K_{2}\to K_{1} \to L_{2}'$ and $K_{2} \to
    L_{1}\to L_{2}'$ give the same map on $\Jsharp$, and
\item the composites $L'_{2}\to K_{0}\to K_{2}$ and $L'_{2}\to
    L_{0}\to K_{2}$ give the same map on $\Jsharp$.
\end{enumerate}
\end{theorem}

\begin{proof}
    The first two items are verbatim restatements of cases of
    Theorem~\ref{thm:intro-one-bar}. The second two cases are cases of
    Theorem~\ref{thm:intro-one-bar} and
    Theorem~\ref{thm:intro-three-bar}. (The pictures are rotated a
    quarter turn relative to the standard pictures. Alternatively,
    these pictures portray the dual of the standard triangles.)

    The composite cobordism $K_{0}\to K_{2}\to K_{1}$ is equal to the
    connect sum $\Sigma\# \Psi_{0}$, where $\Sigma$ is the standard
    cobordism from $K_{0}$ to $K_{1}$ \cite{KM-unknot}. So the
    commutativity in case \ref{item:first-comm-face} follows from
    Proposition~\ref{prop:face-sums}. The next three cases of the
    theorem are similar, except that the connect sums are
    with $\Psi_{2}$ at a tetrahedral point in case
    \ref{item:second-comm-face}, and with $\Psi_{2}$ at a seam point
    in cases  \ref{item:third-comm-face} and
    \ref{item:fourth-comm-face}. So Propositions~\ref{prop:sum-with-Psi-2-3}
    and \ref{prop:seam-sums} deal with these cases.

    In each of the final two cases of the theorem, the first composite
    cobordism is obtained from the second composite by forming a
    connect sum with $\Psi_{2}$ at a tetrahedral point. So these cases
    are also consequences of Proposition~\ref{prop:sum-with-Psi-2-3}.
\end{proof}

\section{Equivalent formulations of the Tutte relation}

The authors conjectured in \cite{KM-jsharp} that if $K$ is a planar
web (i.e. is contained in $\R^{2}\subset \R^{3}$), then the dimension
of $\Jsharp(K)$ is equal to the number of Tait colorings of the
underlying abstract graph. (See Conjecture~\ref{thm:Tait-count} in
\cite{KM-jsharp}.) As explained there, confirming this conjecture
would provide a new proof of Appel and Haken's four-color theorem. If
we write $\tau(K)$ for the number of Tait colorings of $K$, then
$\tau$ is uniquely characterized, for planar webs, by the following
properties:

\begin{enumerate}
\item $\tau(K)=3$ if $K$ is a circle;
\item $\tau(K)=0$ if $K$ has a bridge;
\item $\tau$ is multiplicative for disjoint unions of planar webs;
\item $\tau$ satisfies the ``Tutte relation'', namely if $K_{0}$,
    $K_{1}$, $L_{0}$, $L_{1}$ are planar webs which differ only in a
    ball, in the following manner,
    \[
              K_{0}=\mathfig{0.079}{figures/Graph-Res0}, \qquad
              K_{1}=\mathfig{0.079}{figures/Graph-Res1}, \qquad
              L_{0}=\mathfig{0.079}{figures/Graph-I}, \qquad
              L_{1}=\mathfig{0.079}{figures/Graph-H},
    \]
    then
    \begin{equation}
        \label{eq:Tutte-relation}
         \tau( K_{0} ) -\tau(K_{1}) + \tau(L_{0}) - \tau(L_{1})=0.
    \end{equation}
\end{enumerate}

The first three of these properties hold also for the quantity $\dim
\Jsharp(K)$, for planar webs $K$. They are proved in
\cite{KM-jsharp}. So the question of whether $\dim \Jsharp(K)$ is
equal to the number of Tait colorings is equivalent to the following
conjecture:

\begin{conjecture}\label{conj:Tutte}
    If $K_{0}$, $K_{1}$, $L_{0}$, $L_{1}$ are three planar webs
    differing only in a ball as shown above, then 
\begin{equation}\label{eq:Tutte-Jsharp}
 \dim \Jsharp( K_{0} ) -\dim \Jsharp(K_{1}) 
 + \dim \Jsharp(L_{0}) - \dim \Jsharp(L_{1})=0.
\end{equation}
\end{conjecture}

The four webs that appear in this conjecture appear also as the four
vertices of the central rectangle in the octahedral diagram,
Figure~\ref{fig:J-octahedron}. We reproduce that part of the diagram here:
\begin{equation}\label{eq:Tutte-diagram}
\begin{array}{c}
    \includegraphics[scale=0.46]{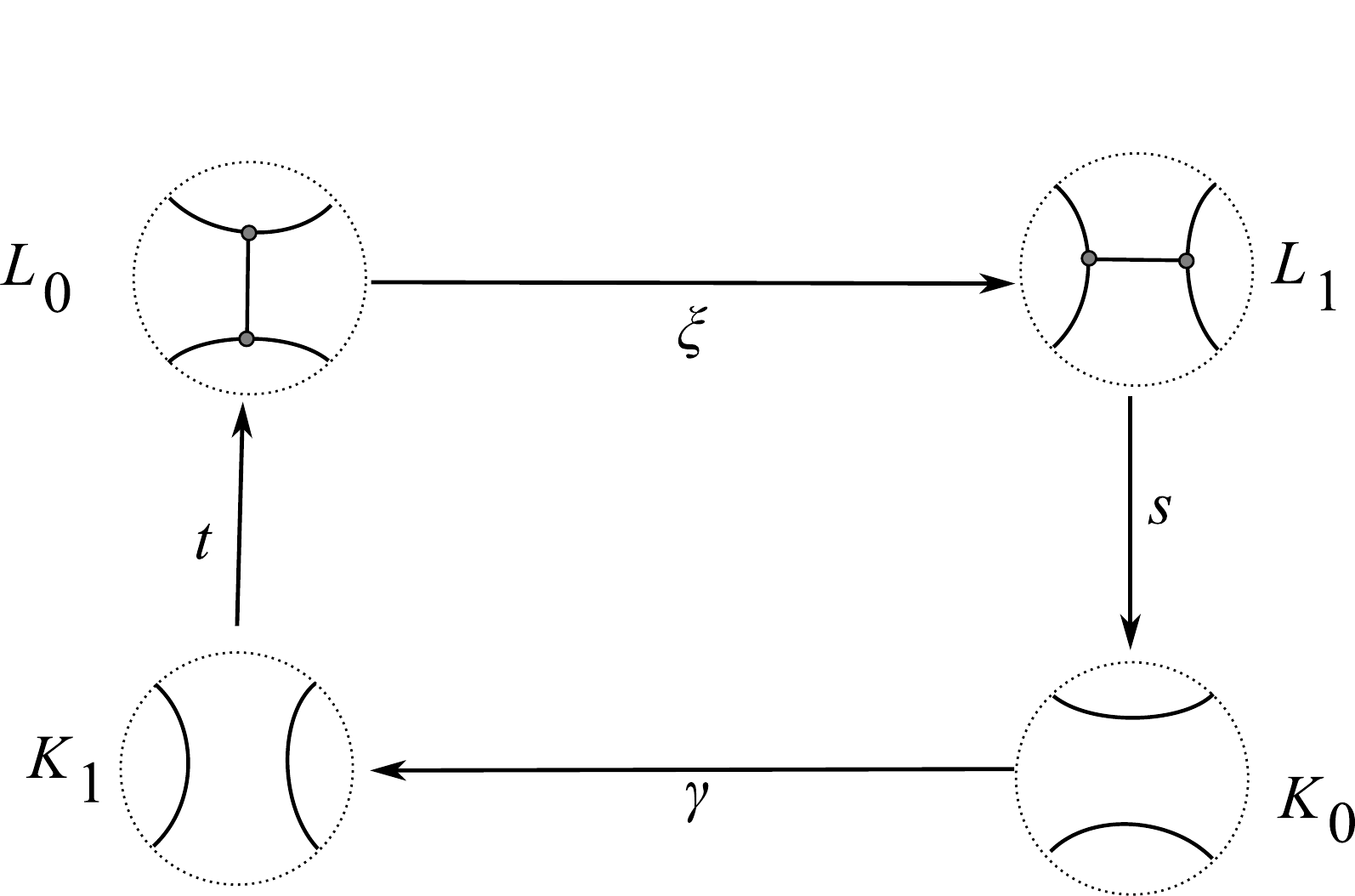} \end{array}
\end{equation}

\begin{lemma}
    In the rectangle above, the composite of any two consecutive foams
    is the zero map on $\Jsharp$. So the vector spaces
    $\Jsharp(K_{0})$, $\Jsharp(K_{1})$, $\Jsharp(L_{0})$,
    $\Jsharp(L_{1})$, together with the maps between them, form a
    chain complex which is periodic mod $4$.
\end{lemma}

\begin{proof}
    Referring to Figure~\ref{fig:J-octahedron} and
    Theorem~\ref{thm:octahedron-statement}, we see that (with the
    application of $\Jsharp$ implied), $t\comp\gamma = t\comp b \comp
    a$, because $\gamma=b\comp a$. On the other hand, $t\comp b=0$,
    because these are two sides of an exact triangle. This shows that
    $t\comp\gamma=0$, and essentially the same argument deals with the
    composites at the other three vertices.
\end{proof}

We can now interpret Conjecture~\ref{conj:Tutte} as asserting that the Euler
characteristic of the $4$-periodic complex \eqref{eq:Tutte-diagram} is zero.
 Since the Euler
characteristic can be computed equally as the alternating sum of the
dimensions of the chain groups or as the alternating sum of the
dimensions of the homology groups, the left-hand side of
\eqref{eq:Tutte-Jsharp} can be expressed also as
\begin{equation}\label{eq:Tutte-Euler}
  \left | \frac{\ker(\gamma) }{ \im(s) }\right  | - 
   \left | \frac{\ker(t) }{ \im(\gamma) }\right | +
   \left | \frac{\ker(\xi) }{ \im(t) } \right |  -
 \left | \frac{\ker(s) }{ \im(\xi) } \right |,
\end{equation}
where $|V|$ denotes the dimension of the vector space $V$, and
$\Jsharp$ is understood.

\begin{lemma}
    In the $4$-periodic complex \eqref{eq:Tutte-diagram}, the homology
   groups at diametrically opposite corners are equal. Furthermore,
   the Euler characteristic \eqref{eq:Tutte-Euler} is equal to
    \[2\bigl(\rank(a\comp \kappa) -
   \rank(\lambda\comp b)\bigr),\] and also equal to \[2\bigl(\rank(a) -
   \rank(b)\bigr).\]
   Here $a$, $b$, $\kappa$ and $\lambda$ are the maps corresponding to
   the standard cobordisms in Figure~\ref{fig:J-octahedron}, and the
   application of $\Jsharp$ is implied.
\end{lemma}

\begin{proof}
    From the exactness of the triangles of maps $(\lambda, \kappa,
    \gamma)$ and $(a, r, s)$, we have $\ker(\gamma)=\im(\kappa)$ and
    $\im(s)=\ker(a)$. So
    \[
               \frac{\ker(\gamma)}{\im(s)} = \frac{\im(\kappa)}{\ker(a)}.
    \]
     At the same time, we learn that $\ker(a)\subset\im(\kappa)$, so
     that the dimension of $\im(\kappa)/\ker(a)$ is equal to the rank
     of $a\comp\kappa$. Thus
     \[
                  \left|  \frac{\ker(\gamma)}{\im(s)}\right| =
                  \rank(a\comp\kappa).
     \]
      A similar discussion can be applied to each of the four terms in
      \eqref{eq:Tutte-Euler}, so the dimensions of the four quotients
      that appear there are respectively
\[ 
                \rank(a \comp \kappa), \quad
                \rank(\lambda \comp b), \quad
                \rank(q \comp \zeta), \quad
                 \rank(\eta \comp r).
\]
The last two parts of Theorem~\ref{thm:octahedron-statement} say that
$\lambda\comp b = \eta\comp r$ and $a\comp\kappa=q\comp\zeta$. So
alternate terms of the above four are equal. This verifies the first
assertion in the lemma, and also shows that the Euler characteristic
is equal to $ 2(\rank(a\comp \kappa) -
   \rank(\lambda\comp b))$.  

From the exact triangle $(a,r,s)$, we have
\[
          2\rank(a) = \left |\mathfig{0.079}{figures/Graph-Res0}\right| - 
                                 \left
                                     |\mathfig{0.079}{figures/Graph-H}\right|
                                   +  \left |\mathfig{0.079}{figures/Graph-X}\right|,                         
\]
while from the exact triangle $(b,t,q)$, we have
\[
         2\rank(b) = \left |\mathfig{0.079}{figures/Graph-Res1}\right| - 
                                 \left
                                     |\mathfig{0.079}{figures/Graph-I}\right|
                                   +  \left |\mathfig{0.079}{figures/Graph-X}\right|.                        
\]
Taking the difference of these last two equalities, we obtain
\[
           2 \rank(a) - 2\rank(b) = 
 \left |\mathfig{0.079}{figures/Graph-Res0}\right| - 
                                 \left
                                     |\mathfig{0.079}{figures/Graph-H}\right|
                  +
                                 \left
                                     |\mathfig{0.079}{figures/Graph-I}\right|
                -  \left |\mathfig{0.079}{figures/Graph-Res1}\right| .
\]
This is the last assertion of the lemma.
\end{proof}

\begin{remarks}
    Since Conjecture~\ref{conj:Tutte} asserts the vanishing of an
    Euler characteristic for planar webs, it is natural to ask whether
    something stronger is true, namely that the $4$-periodic sequence of
    maps \eqref{eq:Tutte-diagram} is exact. By the proof of the above
    lemma, this would be equivalent to the vanishing of the composites
    $a\comp \kappa$ and $\lambda\comp b$. Although this holds in
    simple cases, it does not appear to be true in
    general. Calculations for the case that $L_{0}$ is the
    $1$-skeleton of a dodecahedron (in the natural planar projection)
    suggest that the rank of $a\comp \kappa$ is $5$ in this case
    \cite{KM-dodecahedron-Tutte}. From this it follows (by the lemma)
    that the Euler characteristic of the complex is at most $10$. The
    webs $K_{0}$, $K_{1}$ and $L_{1}$ in this example are ``simple
    webs'' in the sense of \cite{KM-jsharp}, so $\Jsharp$ is easily
    computed for these three. The inequality on the Euler
    characteristic then tells us that the dimension of $\Jsharp$ of
    the dodecahedral graph is at most $70$. In the other direction, a
    different calculation  \cite{KM-dodecahedron-Tutte}
    leads to a lower bound of $58$ on the
    dimension. So we have
\[
                     58 \le \dim \Jsharp(\text{Dodecahedron}) \le 70.
\]
  The number of Tait colorings, on the other hand, is $60$. 

  A slightly sharper upper bound for the dimension of $\Jsharp$ of the
  dodecahedron arises by a different argument. In
  \cite{KM-dodecahedron-rep}, the representation variety
  $\Rep^{\sharp}(K)$ is described for the dodecahedron: it consists
  of $10$ copies of the flag manifold and two copies of $\SO(3)$. The
  Chern-Simons functional is Morse-Bott along $\Rep^{\sharp}(K)$, and
  it follows that there is a perturbation of the Chern-Simons
  functional having exactly $68$ critical points, leading to a
  bound of $68$ on the rank of $\Jsharp$.  From this point of view, 
  the question of whether the rank is
  strictly less than $68$ is therefore the question of whether there
  are any non-trivial differentials in the complex, for this
  particular perturbation.
\end{remarks}

\bibliographystyle{abbrv}
\bibliography{fcs}

\begin{thebibliography}{1}

\bibitem{Abreu}
M.~Abreu.
\newblock K\"ahler metrics on toric orbifolds.
\newblock {\em J. Differential Geom.}, 58(1):151--187, 2001.

\bibitem{AHS}
M.~F. Atiyah, N.~J. Hitchin, and I.~M. Singer.
\newblock Self-duality in four-dimensional {R}iemannian geometry.
\newblock {\em Proc. Roy. Soc. London Ser. A}, 362(1711):425--461, 1978.

\bibitem{Bryant}
R.~L. Bryant.
\newblock Bochner-{K}\"ahler metrics.
\newblock {\em J. Amer. Math. Soc.}, 14(3):623--715 (electronic), 2001.

\bibitem{KMOS}
P.~Kronheimer, T.~Mrowka, P.~Ozsv{\'a}th, and Z.~Szab{\'o}.
\newblock Monopoles and lens space surgeries.
\newblock {\em Ann. of Math. (2)}, 165(2):457--546, 2007.

\bibitem{KM-dodecahedron-Tutte}
P.~B. Kronheimer and T.~S. Mrowka.
\newblock {Foam calculations for the $\SO(3)$ instanton homology of the
  dodecahedron}.
\newblock In preparation.

\bibitem{KM-jsharp}
P.~B. Kronheimer and T.~S. Mrowka.
\newblock {Tait colorings, and an instanton homology for webs and foams}.
\newblock Preprint, 2015.

\bibitem{KM-dodecahedron-rep}
P.~B. Kronheimer and T.~S. Mrowka.
\newblock {The $\SO(3)$ representation variety for the dodecahedral web}.
\newblock In preparation.

\bibitem{KM-unknot}
P.~B. Kronheimer and T.~S. Mrowka.
\newblock Khovanov homology is an unknot-detector.
\newblock {\em Publ. Math. IHES}, 113:97--208, 2012.

\bibitem{OS-double-covers}
P.~Ozsv{\'a}th and Z.~Szab{\'o}.
\newblock On the {H}eegaard {F}loer homology of branched double-covers.
\newblock {\em Adv. Math.}, 194(1):1--33, 2005.

\end{thebibliography}

\end{document}